\documentclass[10pt]{article}
\usepackage{amsmath,amssymb,amsthm}
\usepackage{epsfig,psfrag,url,color}
\usepackage{graphicx,subfigure}
\usepackage[symbol]{footmisc}
\usepackage[utf8]{inputenc}
\usepackage{subfigure,enumerate}

\let\eps=\eta

\newcommand{\cond}[1]{Condition~\ref{cond:#1}}
\numberwithin{equation}{section}

\newcommand{\bigo}{{\mathcal O}}

\newcommand{\half}{\frac{1}{2}}

\def\XXint#1#2#3{{\setbox0=\hbox{$#1{#2#3}{\int}$}
     \vcenter{\hbox{$#2#3$}}\kern-.5\wd0}}

\DeclareMathOperator{\diag}{diag}

\DeclareMathOperator{\imag}{Im}

\DeclareMathOperator{\real}{Re}

\newenvironment{mat}{\left[\begin{array}{ccccccccccccccc}}{\end{array}\right]}
\newcommand\bcm{\begin{mat}}
\newcommand\ecm{\end{mat}}

\newcommand{\bea}{\begin{eqnarray}}
\newcommand{\eea}{\end{eqnarray}}
\newcommand{\bean}{\begin{eqnarray*}}
\newcommand{\eean}{\end{eqnarray*}}
\newcommand{\ba}{\begin{array}}
\newcommand{\ea}{\end{array}}
\newcommand{\beqs}{\begin{equation*}\begin{split}}

\newcommand{\lem}[1]{Lemma~\ref{#1}}

\newtheorem{example}{Example}[section]

\newtheorem{remark}{Remark}[section]

\newtheorem{lemma}{Lemma}[section]

\newtheorem{theorem}{Theorem}[section]

\newtheorem{definition}{Definition}[section]

\newtheorem{corollary}{Corollary}[section]

\newtheorem{proposition}{Proposition}[section]

\newtheorem{condition}{Condition}[section]

\long\def\symbolfootnote[#1]#2{\begingroup%
\def\thefootnote{\fnsymbol{footnote}}\footnote[#1]{#2}\endgroup}

\newcommand{\D}{\mathrm{d}}
\newcommand{\I}{\mathrm{i}}
\newcommand{\E}{\mathrm{e}}

\newcommand{\ds}{\displaystyle}

\begin{document}
\renewcommand*{\thefootnote}{\fnsymbol{footnote}}
\title{Universality for the Toda algorithm to compute the largest eigenvalue of a random matrix}
\author{Percy Deift$^*$ and Thomas Trogdon$^\dagger$\\
\phantom{.}\\
$^*$Courant Institute of Mathematical Sciences\\
New York University\\
251 Mercer St.\\
New York, NY 10012, USA\\
\phantom{.}\\
$^\dagger$Department of Mathematics\\
University of California, Irvine\\
Rowland Hall\\
Irvine, CA 92697-3875, USA
}
\maketitle

\newcommand{\blue}[1]{{\color{blue}#1}}

\renewcommand{\blue}[1]{#1}


\footnotetext[1]{Email: deift@cims.nyu.edu.}
\footnotetext[2]{Email: ttrogdon@math.uci.edu (corresponding author).}
\vspace{-.2in}

\begin{abstract}
\blue{We prove universality for the fluctuations of the halting time for the Toda algorithm to compute the largest eigenvalue of real symmetric and complex Hermitian matrices.}  The proof relies on recent results on the statistics of the eigenvalues and eigenvectors of random matrices (such as delocalization, rigidity and edge universality) in a crucial way.
\end{abstract}


\renewcommand*{\thefootnote}{\arabic{footnote}}
\section{Introduction}
In \cite{DiagonalRMT} the authors initiated a statistical study of the performance of various standard algorithms $\mathcal A$ to compute the eigenvalues of random real symmetric matrices $H$. Let $\Sigma_N$ denote the set of real $N \times N$ symmetric matrices.  Associated with each algorithm $\mathcal A$, there is, in the discrete case such as QR, a map $\varphi = \varphi_{\mathcal A}: \Sigma_N \to \Sigma_N$, with the properties
\begin{itemize}
\item (isospectral) $\mathrm{spec}(\varphi_{\mathcal A}(H)) = \mathrm{spec(H)}$,
\item (convergence) the iterates $X_{k+1} = \varphi_{\mathcal A}(X_k)$, $k \geq 0$, $X_0 = H$ given, converge to a diagonal matrix $X_\infty$, $X_k \to X_\infty$ as $k \to \infty$,
\end{itemize}
and in the continuum case, such as Toda, there is a flow $t \mapsto X(t) \in \Sigma_N$ with the properties
\begin{itemize}
\item (isospectral) $\mathrm{spec}(X(t))$ is constant,
\item (convergence)  the flow $X(t)$, $t \geq 0$,  $X(0) = H$ given, converges to a diagonal matrix $X_\infty$, $X(t) \to X_\infty$ as $t \to \infty$.
\end{itemize}
In both cases, necessarily, the (diagonal) entries of $X_\infty$ are the eigenvalues of the given matrix $H$.

Given $\epsilon > 0$, it follows, in the discrete case, that for some $m$ the off-diagonal entries of $X_m$ are $\mathcal O(\epsilon)$ and hence the diagonal entries of $X_m$ give the eigenvalues of $X_0 = H$ to $\mathcal O(\epsilon)$.  The situation is similar for continuous algorithms $t \mapsto X(t)$.  Rather than running the algorithm until all the off-diagonal entries are $\mathcal O (\epsilon)$, it is customary to run the algorithm with \emph{deflations} as follows.  For an $N \times N$ matrix $Y$ in block form
\begin{align*}
  Y = \begin{mat} Y_{11} & Y_{12} \\ Y_{21} & Y_{22} \end{mat},
\end{align*}
with $Y_{11}$ of size $k\times k$ and $Y_{22}$ of size $N -k \times N- k$ for some $k \in \{1,2,\ldots,N-1\}$, the process of projecting $Y \mapsto \diag(Y_{11},Y_{22})$ is called deflation.  For a given $\epsilon$, algorithm $\mathcal A$ and matrix $H \in \Sigma_N$, define the \emph{$k$-deflation} time $T^{(k)}(H) = T_{\epsilon,\mathcal A}^{(k)}(H)$, $1 \leq k \leq N-1$, to be the smallest value of $m$ such that $X_m$, the $m$th iterate of algorithm $\mathcal A$ with $X_0 = H$, has block form
\begin{align*}
X_m = \begin{mat} X^{(k)}_{11} & X^{(k)}_{12}\\
X^{(k)}_{21} & X^{(k)}_{22} \end{mat},
\end{align*}
with $X_{11}^{(k)}$ of size $k \times k$ and $X_{22}^{(k)}$ of size $N-k \times N-k$ and\footnote{Here we use $\|\cdot\|$ to denote the Frobenius norm $\|X\|^2 = \sum_{i,j} |X_{ij}|^2$ for $X = (X_{ij})$.} $\|X_{12}^{(k)}\| = \|X_{21}^{(k)}\| \leq \epsilon$.  The defation time $T(H)$ is then defined as
\begin{align*}
  T(H) = T_{\epsilon,\mathcal A}(H) = \min_{1 \leq k \leq N-1} T_{\epsilon,\mathcal A}^{(k)}(H).
\end{align*}
If $\hat k \in \{1,\ldots,N-1\}$ is such that \blue{$T(H) = T_{\epsilon,\mathcal A}^{(\hat k)}(H)$}, it follows that the eigenvalues of $H = X_0$ are given by the eigenvalues of the block-diagonal matrix $\diag(X_{11}^{(\hat k)}, X_{22}^{(\hat k)})$ to $\mathcal O(\epsilon)$.  After, running the algorithm to time $\blue{T_{\epsilon,A}(H)}$, the algorithm restarts by applying the basic algorithm $\mathcal A$ separately to the smaller matrices $X_{11}^{(\hat k)}$ and $X_{22}^{(\hat k)}$ until the next deflation time, and so on.  There are again similar considerations for continuous algorithms.

As the algorithm proceeds, the number of matrices after each deflation doubles.  This is counterbalanced by the fact that the matrices are smaller and smaller in size, and the calculations are clearly parallelizable.  Allowing for parallel computation, the number of deflations to compute all the eigenvalues of a given matrix $H$ to a given accuracy $\epsilon$, will vary from $\mathcal O(\log N)$ to $\mathcal O(N)$.

In \cite{DiagonalRMT}  the authors considered the deflation time $T= T_{\epsilon,\mathcal A}$ for $N \times N$ matrices chosen from a given ensemble $\mathcal E$.  \blue{Henceforth in this paper we suppress the dependence on $\epsilon, N, \mathcal A$ and $\mathcal E$, and simply write $T$ with these variables understood.}   For a given algorithm $\mathcal A$ and ensemble $\mathcal E$ the authors computed $T(H)$ for 5,000-15,000 samples of matrices $H$ chosen from $\mathcal E$, and recorded the \emph{normalized deflation time}
\begin{align}\label{e:ndt}
\tilde T(H) :=  \frac{T(H) - \langle T\rangle }{\sigma},
\end{align}
where $\langle T\rangle$ and $\sigma^2 = \langle (T - \langle  T \rangle)^2 \rangle$ are the sample average and sample variance of $T(H)$, respectively.  Surprisingly, the authors found that for the given algorithm $\mathcal A$, and $\epsilon$ and $N$ in a suitable scaling range with $N \to \infty$, the histogram of $\tilde T$ was universal, independent of the ensemble  $\mathcal E$.  In other words, the fluctuations in the deflation time $\tilde T$, suitably scaled, were universal, independent of $\mathcal E$.
\begin{figure}[tbp]
\centering
\subfigure[]{\includegraphics[width=.4\linewidth]{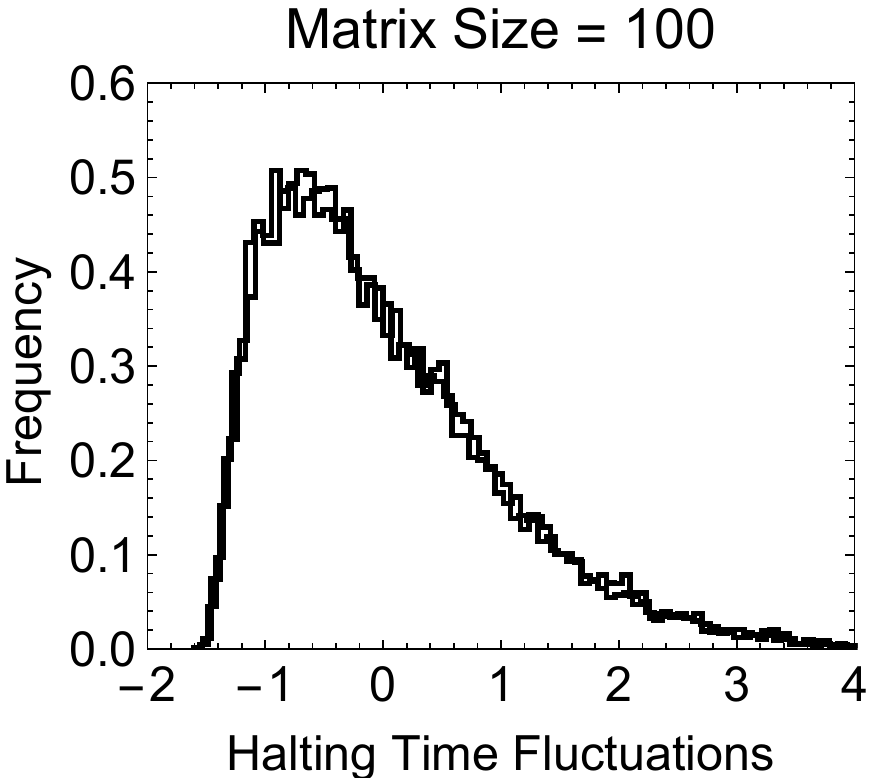}\label{f:PDM-QR}}
\hspace{.3in}\subfigure[]{\includegraphics[width=.4\linewidth]{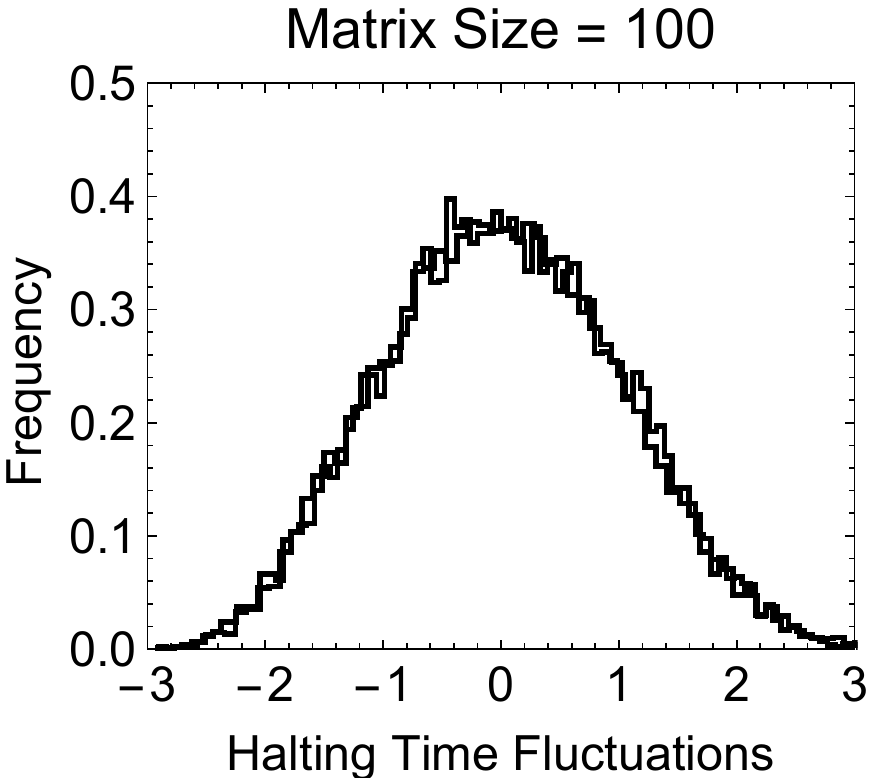}\label{f:PDM-Toda} }
\caption{\label{f:PDM} Universality for $\tilde T$ when (a) $\mathcal A$ is the QR eigenvalue algorithm and when (b) $\mathcal A$ is the Toda algorithm.  Panel (a) displays the overlay of two histograms for $\tilde T$ in the case of QR, one for each of the two ensembles $ \mathcal E = \mathrm{BE}$, consisting of iid mean-zero Bernoulli random variables (see Definition~\ref{def:WE}) and $\mathcal E = \mathrm{GOE}$, consisting of iid mean-zero normal random variables. Here $\epsilon= 10^{-10}$ and $N = 100$.  Panel (b) displays the overlay of two histograms for $\tilde T$ in the case of the Toda algorithm, and again $\mathcal E = \mathrm{BE}$ or $\mathrm{GOE}$. And here $\epsilon= 10^{-8}$ and $N = 100$.}
\end{figure}
 Figure~\ref{f:PDM} displays some of the numerical results from \cite{DiagonalRMT}.  Figure~\ref{f:PDM-QR} displays data for the QR algorithm, which is discrete, and Figure~\ref{f:PDM-Toda} displays data for the Toda algorithm, which is continuous.

Subsequently, in \cite{Deift2014}, the authors raised the question of whether the universality results in \cite{DiagonalRMT} were limited to eigenvalue algorithms for real symmetric matrices, or whether they were present more generally in numerical computation.  And indeed the authors in \cite{Deift2014} found similar universality results for a wide variety of numerical algorithms, including
\begin{itemize}
\item \blue{other algorithms such as the QR algorithm with shifts, the Jacobi eigenvalue algorithm, and also algorithms applied to complex Hermitian ensembles,}
\item the conjugate gradient and GMRES algorithms to solve linear $N\times N$ systems $Hx = b$,
\item an iterative algorithm to solve the Dirichlet problem $\Delta u = 0$ on a random star-shaped region $\Omega \subset \mathbb R^2$ with random boundary data $f$ on $\partial \Omega$, and
\item a genetic algorithm to compute the equilibrium measure for orthogonal polynomials on the line.
\end{itemize}

\blue{All of the above results are numerical. The goal of this paper is to establish universality as a bona fide phenomenon in numerical analysis, and not just an artifact, suggested, however strongly, by certain computations as above.  To this end we seek out and prove universality for an algorithm of interest.  We focus, in particular, on eigenvalue algorithms.} \blue{To analyze eigenvalue algorithms with deflation, one must first analyze $T^{(k)}$ for $1 \leq k \leq N-1$, and then compute the minimum of these $N-1$ dependent variables. The analysis of $T^{(k)}$ for $1 \leq k \leq N-1$ requires very detailed information on the eigenvalues and eigenvectors of random matrices that, at this time, has only been established for $T^{(1)}$ (see below).  Computing the minimum requires knowledge of the distribution of $\hat k$ such that $T(H) = T^{(\hat k)}(H)$, which is an analytical problem that is still untouched.  In Figure~\ref{f:hatk} we show the statistics of $\hat k$ obtained numerically for the Toda algorithm\footnote{A similar histogram for the QR algorithm has an asymmetry that reflects the fact that typically $H$ has an eigenvalue near zero: For QR, a simple argument shows that eigenvalues near zero favor $\hat k = N-1$.}.  In view of the above issues, a comprehensive analysis of the algorithms with deflation, seems, currently, to be out of reach.  In this paper we restrict our attention to the Toda algorithm, and as a \underline{first step} towards understanding $T(H)$ we prove universality for the fluctuations of $T^{(1)}(H)$, the $1$-deflation time for Toda --- see Theorem~\ref{t:main}.   As we see from Proposition~\ref{p:error}, with high probability $X_{11}(T^{(1)}) \sim \lambda_N$, the largest eigenvalue of $X(0) = H$.   In other words, $T^{(1)}(H)$ controls the computation of the largest eigenvalue of $H$ via the Toda algorithm.  Theorem~\ref{t:main} and Proposition~\ref{p:error} are the main results in this paper.  Much of the detailed statistical information on the eigenvalues and eigenvectors of H needed to analyze $T^{(1)}(H)$, was only established in the last 3 or 4 years.}

%
\begin{figure}[tbp]
\centering
  \includegraphics[width=.45\linewidth]{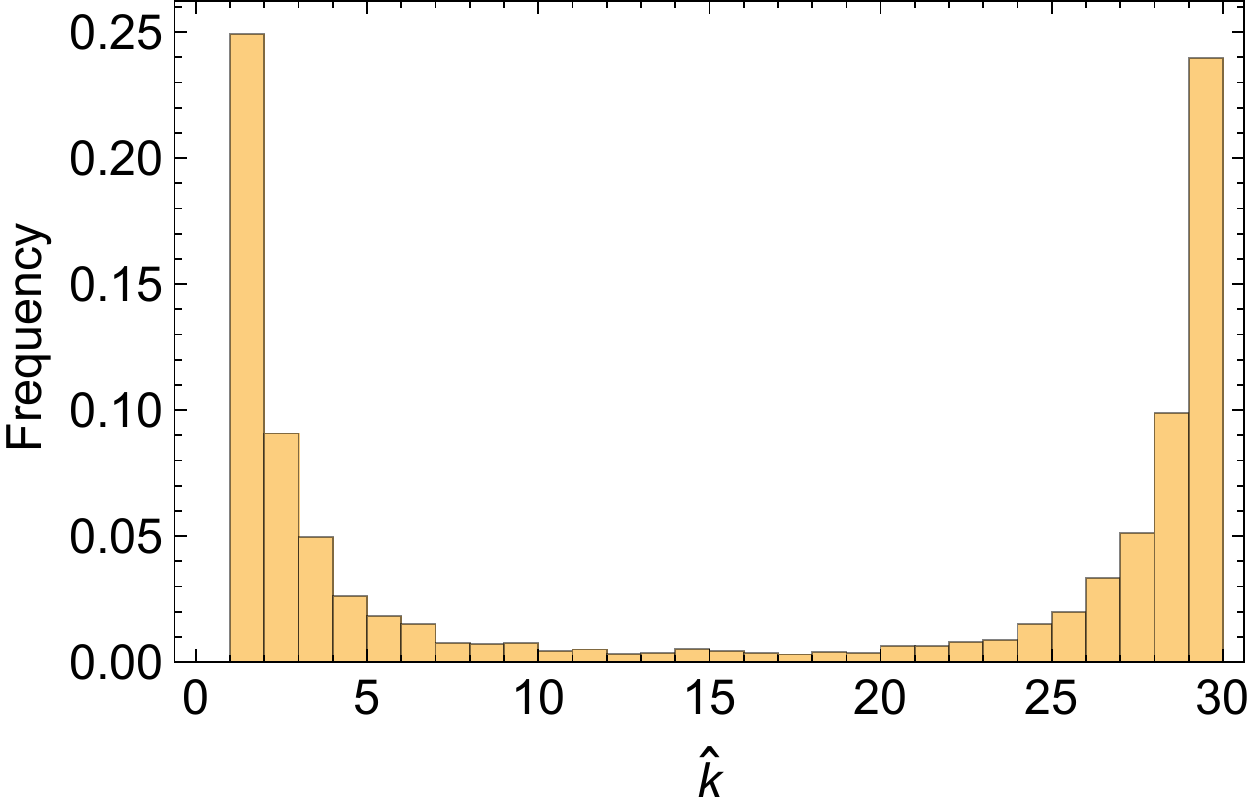}
  \caption{\label{f:hatk}  \blue{The distribution of $\hat k$ for GOE when $N = 30$, $\epsilon = 10^{-8}$ for the Toda algorithm: $\hat k = 1,N-1$ are equally likely.}}
\end{figure}
%



\blue{In this paper we always order the eigenvalues $\lambda_{n} \leq \lambda_{n+1}$, $n = 1, \ldots, N$.}  In Sections~\ref{s:main-results} and \ref{s:rmt} we will describe some of the properties of the Toda algorithm and some results from random matrix theory.  In Section~\ref{s:numerics} we describe some numerical results demonstrating Theorem~\ref{t:main}.  Note that Figure~\ref{f:TodaOne} for $T^{(1)}(H)$ is very different from Figure~\ref{f:PDM-Toda} for $T(H)$.  In Sections~\ref{s:universality} and \ref{s:prob} we will prove universality for $T^{(1)}$ for matrices from generalized Wigner ensembles and also from invariant ensembles.  \blue{See Appendix~\ref{s:ensemble} for a full description of these random matrix ensembles.  The techniques in this paper can also be used to prove universality for the fluctuations in the halting times for other eigenvalue algorithms, in particular, QR (without shifts) --- see Remark~\ref{r:other} below.}

\subsection{Main result}\label{s:main-results}
The Toda algorithm is an example of the generalized continuous eigenvalue algorithm described above. For an $N\times N$ real symmetric or Hermitian matrix $X(t) = (X_{ij}(t))_{i,j = 1}^N$, the Toda equations are given by\footnote{In the real symmetric case $^*$ should be replaced with $^T$.}
\begin{align}\label{e:Toda}
\dot X &= \left[X, B(X)\right], \quad B(X) = X_- - (X_-)^*, \quad X(0) = H = H^*,
\end{align}
where $X_-$ is the strictly lower-triangular part of $X$ and $\left[A,B\right]$ is the standard matrix commutator. It is well known that this flow is isospectral and converges as $t \to \infty$ to a diagonal matrix $X_\infty= \diag(\lambda_N,\ldots,\lambda_1)$; see for example \cite{Deift1985}. As noted above, necessarily, the diagonal elements of $X_\infty$ are the eigenvalues of $H$.  By the \textbf{Toda algorithm} to compute the eigenvalues of a Hermitian matrix $H$ we mean solving \eqref{e:Toda} with $X(0) = H$ until such time $t'$ that the off-diagonal elements in the matrix $X(t)$ are of order $\epsilon$.  The eigenvalues of $X(t')$ then give the eigenvalues of $H$ to $\mathcal O(\epsilon)$.

The history of the Toda algorithm is as follows. The Toda Lattice was introduced by M. Toda in 1967 \cite{Toda1967} and describes the motion of $N$ particles $x_i$, $i=1,\ldots,N$, on the line under the Hamiltonian
\begin{align*}
H_{\mathrm{Toda}}(x,y) = \half \sum_{i=1}^N y_i^2 + \half \sum_{i=1}^N \E^{x_i - x_{i+1}}.
\end{align*}
In 1974, Flaschka \cite{Flaschka1974} (see also \cite{Manakov1975}) showed that Hamilton's equations
\begin{align*}
\dot x = \frac{\partial H_{\mathrm{Toda}}}{\partial y}, \quad \dot y = -\frac{\partial H_{\mathrm{Toda}}}{\partial x},
\end{align*}
can be written in the Lax pair form \eqref{e:Toda} where $X$ is tridiagonal
\begin{align*}
X_{ii} &= -y_i/2, \quad 1 \leq i \leq N, \\
X_{i,i+1} &= X_{i+1,i} = \half \E^{\half(x_i - x_{i+1})}, \quad 1 \leq i \leq N-1,
\end{align*}
and $B(X)$ is the tridiagonal skew-symmetric matrix $B(X) = X_- - (X_-)^T$ as in \eqref{e:Toda}.  As noted above, the flow $t \mapsto X(t)$ is isospectral.  But more is true: The flow is completely integrable in the sense of Liouville with the eigenvalues of $X(0)=H$ providing $N$ Poisson commuting integrals for the flow.  In 1975, Moser showed that the off-diagonal elements $X_{i,i+1}(t)$ converge to zero as $t \to \infty$ \cite{Moser1975}.  Inspired by this result, and also related work of Symes \cite{Symes1982} on the QR algorithm, the authors in \cite{DeiftEigenvalue} suggested that the Toda Lattice be viewed as an eigenvalue algorithm, the Toda algorithm. The Lax equations \eqref{e:Toda} clearly give rise to a global flow not only on tridiagonal matrices but also on general real symmetric matrices. It turns out that in this generality \eqref{e:Toda} is also Hamiltonian \cite{Kostant1979,Adler1978} and, in fact, integrable \cite{Deift1986a}.  From that point on, by the Toda algorithm one means the action of \eqref{e:Toda} on full real symmetric matrices, or by extension, on complex Hermitian matrices.\footnote{The Toda flow \eqref{e:Toda} also generates a completely integrable Hamiltonian system on real (not necessarily symmetric) $N \times N$ matrices, see \cite{Deift1985}. The Toda flow \eqref{e:Toda} on Hermitian matrices was first investigated by Watkins \cite{WatkinsIsospectral}. }

As noted in the Introduction, in this paper we consider running the Toda algorithm only until time $T^{(1)}$, the deflation time with block decomposition $k = 1$ fixed, when the norm of the off-diagonal elements in the first row, and hence the first column, is $\mathcal O(\epsilon)$. Define
\begin{align}\label{e:H-rowsum}
E(t) = \sum_{n=2}^N |X_{1n}(t)|^2,
\end{align}
so that if $E(t) = 0$ then $X_{11}(t)$ is an eigenvalue of $H$.  Thus, with $E(t)$ as in \eqref{e:H-rowsum}, the halting time (or $1$-deflation time) for the Toda algorithm is given by
\begin{align}\label{e:toda-halt}
T^{(1)}(H) = \inf\{t: E(t) \leq \epsilon^2\}.
\end{align}
Note that by the min-max principle if $E(t) < \epsilon^2$ then $|X_{11}(t)-\lambda_j| < \epsilon$ for some eigenvalue $\lambda_j$ of $X(0)$.
%
\begin{figure}[tbp]
\centering
\includegraphics[width=.6\linewidth]{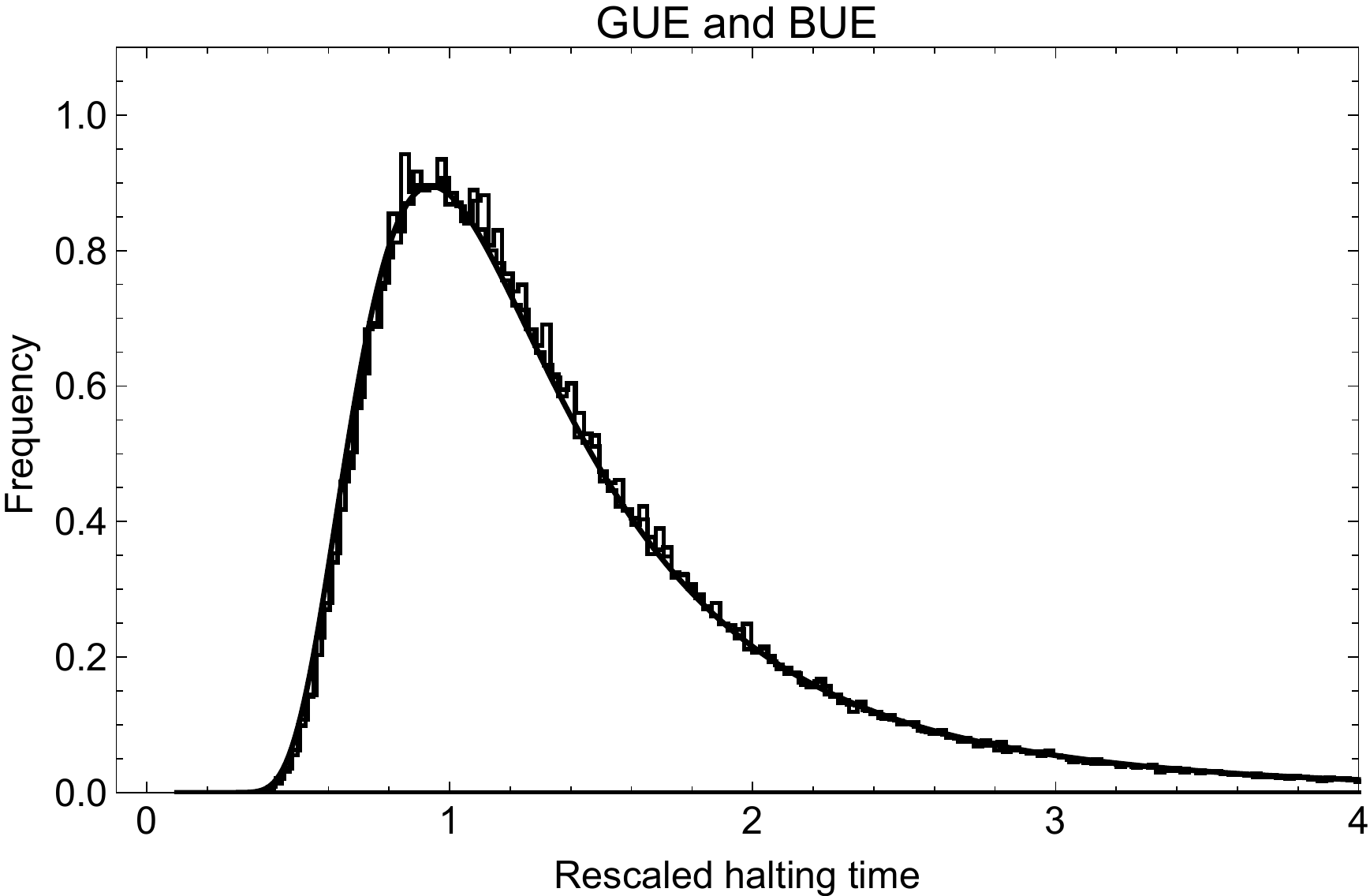}
\caption{\label{f:TodaOne}
The simulated rescaled histogram for $\tilde T^{(1)}$ for both BUE and GUE.  Here $\epsilon = 10^{-14}$ and $N = 500$ with 250,000 samples. The solid curve is the rescaled density $f_2^\mathrm{gap}(t) = \D/\D t F_{2}^\mathrm{gap}(t)$.  The density $f_2^\mathrm{gap}(t) = \frac{1}{\sigma t^2}A^\mathrm{soft}\left(\frac{1}{\sigma t}\right)$, where $A^\mathrm{soft}(s)$ is shown in \cite[Figure 1]{Witte2013}: In order to match the scale in \cite{Witte2013} our choice of distributions (BUE and GUE) we must take $\sigma = 2^{-7/6}$.  This is a numerical demonstration of Theorem~\ref{t:main}.}
\end{figure}
\begin{figure}[tbp]
\centering
\includegraphics[width=.6\linewidth]{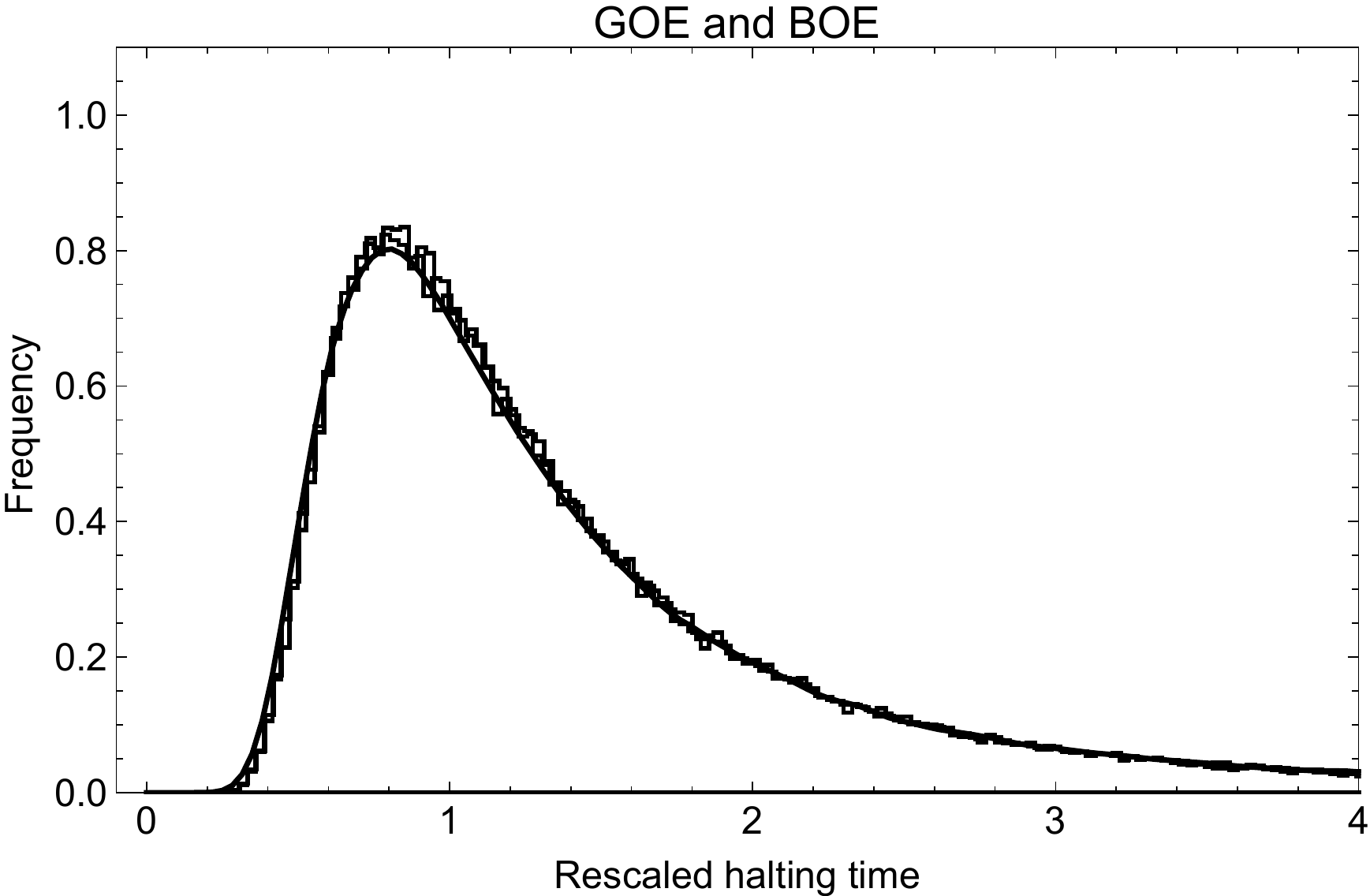}
\caption{\label{f:TodaOneO}
The simulated rescaled histogram for $\tilde T^{(1)}$ for both BOE and GOE demonstrating Theorem~\ref{t:main}. Here $\epsilon = 10^{-14}$ and $N = 500$ with 250,000 samples. The solid curve is an approximation to the density $f_1^\mathrm{gap}(t) = \D/\D t F_{1}^\mathrm{gap}(t)$.  We compute $f_1^\mathrm{gap}(t)$ by smoothing the histogram for $c_V^{-2/3}2^{-2/3}N^{-2/3}(\lambda_{N}-\lambda_{N-1})$ when $N = 800$ with 500,000 samples.}
\end{figure}


\blue{For invariant and generalized Wigner random matrix ensembles there is a constant $c_V$, which depends on the ensemble, such that the following limit exists ($\beta = 1$ for the real symmetric case, $\beta = 2$ for the complex Hermitian case)
\begin{align}\label{e:fb1}
F^\mathrm{gap}_\beta(t) = \lim_{N \to \infty} \mathbb P\left( \frac{1}{c_V^{2/3} 2^{-2/3}N^{2/3}(\lambda_N-\lambda_{N-1})} \leq t \right), \quad t \geq 0.
\end{align}
The precise value of $c_V$ is described in Theorem~\ref{def:EM} and this limit is discussed further in Definition~\ref{def:Fbeta}.  For fixed $\beta$, the limit is independent of the choice of ensemble.}

\blue{\begin{definition}[Scaling region]\label{def:scaling}
Fix $0 < \sigma < 1$.  The {\bf scaling region}\footnote{\blue{From the statement of the theorem, it is reasonable to ask if $5/3$ can be replaced with $2/3$ in the definition of the scaling region.  Also, one should expect different limits for larger values of $\epsilon$ as other eigenvalues will contribute. These questions have yet to be explored. } }  for $(\epsilon,N)$ is given by $\displaystyle \frac{\log{\epsilon^{-1}}}{\log N} \geq 5/3 + \sigma/2$.
\end{definition}}

\blue{Note that for $\epsilon = 10^{-15}$, a relevant value for double-precision arithmetic, $(\epsilon,N)$ is in the scaling region for all values of $N$ less than $10^9$.}

\begin{theorem}[Universality for $T^{(1)}$]\label{t:main}
   Let $0 < \sigma < 1$ be fixed and let $(\epsilon,N)$ be in the scaling region $\frac{\log \epsilon^{-1}}{\log N} \geq \frac{5}{3} + \frac{\sigma}{2}$. Then if $H$ is distributed according to any real ($\beta = 1$) or complex ($\beta = 2$) invariant or Wigner ensemble we have
  \begin{align}\label{e:main-thm}
    \lim_{N\to \infty} \mathbb P\left( \frac{T^{(1)}}{c_V^{2/3} 2^{-2/3}N^{2/3}(\log \epsilon^{-1}-2/3 \log N)} \leq t\right) =  F^{\mathrm{gap}}_\beta(t).
  \end{align}
  Here $c_V$ is the same constant as in \eqref{e:fb1}.
\end{theorem}

\blue{

\begin{example} Consider the case of real symmetric, $2 \times 2$ matrices. For $X(0) = H$, it follows that as $t \to \infty$, $X_{11}(t) \to \lambda_2$, the largest eigenvalue, while $X_{22}(t) \to \lambda_1$, the second-largest eigenvalue.  And so, one should expect $T^{(1)}$ to be larger for
\begin{align*}
X(0) = H_+ := \begin{bmatrix} -1 & \delta \\ \delta &1\end{bmatrix} \quad \text{ than for }\quad X(0) = H_- := \begin{bmatrix} 1 & \delta \\ \delta & -1 \end{bmatrix},
\end{align*}
despite the fact that these matrices have the same eigenvalues.  Said differently, it is surprising that the fluctuations of $T^{(1)}$ in Theorem~\ref{t:main} depend only on the eigenvalues and are independent of the eigenvectors of $H$. Let $U  = (U_{ij})_{1\leq i,j \leq 2}$ be the matrix of normalized eigenvectors of $X(0)$.   It then follows from the calculations in Section~\ref{s:universality} that
\begin{align*}
|X_{12}(t)|^2 &= (\lambda_2-\lambda_1)^2\frac{|U_{11}(0)|^2\E^{2 \lambda_1 t}}{|U_{11}(0)|^2 \E^{2 \lambda_1 t} + |U_{12}(0)|^2 \E^{2 \lambda_2 t}}.
\end{align*}
It is then clear that
\begin{align*}
|X_{12}(t)|^2 \sim (\lambda_2-\lambda_1)^2 \frac{|U_{11}(0)|^2}{ |U_{12}(0)|^2}\E^{-2 (\lambda_2-\lambda_1) t}, \text{ as } ~ t \to \infty.
\end{align*}
First, one should note that this, roughly speaking, explains the appearance of $\lambda_N - \lambda_{N-1}$ in the definition of the universal limit $F_\beta^{\mathrm{gap}}(t)$.  Second, a simple calculation shows that as $\delta \downarrow 0$,  $|U_{12}(0)| \sim \delta$ for $H_+$ while $|U_{12}(0)| \sim 1$ for $H_-$, explaining why $T^{(1)}(H_+) \geq T^{(1)}(H_-)$.  However, the matrices $H_+$ and $H_-$ are not ``typical''.  With high probability, the eigenvectors of random matrices in the ensembles under consideration are delocalized, so that $U_{1j}$, $j =1,\ldots,N$ are all of the same order.  For general $N$, we then have $ \sum_{k = 2}^N |X_{1k}|^2 \asymp (\lambda_N-\lambda_{N-1})^2 \E^{-2(\lambda_{N-1}-\lambda_N)t}$ and the dependence on the eigenvectors is effectively removed as $\epsilon \downarrow 0$.

\end{example}

}

To see that the algorithm computes the top eigenvalue, to an accuracy beyond its fluctuations, we have the following proposition which is a restatement of Proposition~\ref{p:error-alpha} that shows our error is $\bigo(\epsilon)$ with high probability.
\begin{proposition}[Computing the largest eigenvalue]\label{p:error}
   Let $(\epsilon, N)$ be in the scaling region. Then if $H$ is distributed according to any real or complex invariant or Wigner ensemble
  \begin{align*}
  \epsilon^{-1} |\lambda_N - X_{11}(T^{(1)})|
  \end{align*}
  converges to zero in probability as $N \to \infty$.  Furthermore, both
  \begin{align*}
  \epsilon^{-1} |b_V - X_{11}(T^{(1)})|, \quad \epsilon^{-1} |\lambda_j - X_{11}(T^{(1)})|
  \end{align*}
  converge to $\infty$ in probability for any $j = j(N) < N$ as $N \to \infty$, where $b_V$ is the supremum of the support of the equilibrium measure for the ensemble.
\end{proposition}

The relation of this theorem to two-component universality, \blue{as discussed in \cite{Deift2014}}, is the following. Let $\xi = \xi_\beta$ be the random variable with distribution $F^{\mathrm{gap}}_\beta(t)$, $\beta = 1$ or $2$. For $\beta =2$ IEs one can prove that\footnote{We can also prove \eqref{e:L1} for $\beta =1$ IEs.  These proofs of these facts require an extension of the level repulsion estimates in \cite[Theorem 3.2]{Bourgade2014} to the case `$K = 1$'. When $\beta =2$, again with this extension of \cite[Theorem 3.2]{Bourgade2014}  to the case `$K=1$', we can prove that $\kappa = \mathrm{Var}(\xi)$.  This extension is known to be true \cite{BourgadeConv}. The calculations in Table~\ref{t:table} below are consistent with \eqref{e:L1} and \eqref{e:L2} (even for WEs) and lead us to believe that \eqref{e:L2} also holds for $\beta =1$.  Note that for $\beta = 2$, $\mathbb E[\xi^2] < \infty$, but it is believed that $\mathbb E[\xi^2] = \infty$ for $\beta = 1$, see \cite{Perret2013}.  In other words, we face the unusual situation where the variance seems to converge, but not to the variance of the limiting distribution.
}
\begin{align}\label{e:L1}
\mathbb E[T^{(1)}] &= c_V^{2/3}2^{-2/3}N^{2/3}(\log \epsilon^{-1}-2/3 \log N) \mathbb E[\xi](1 + o(1)),\\
\sqrt{\mathrm{Var}(T^{(1)})} &= \kappa c_V^{2/3}2^{-2/3}  N^{2/3}(\log \epsilon^{-1}-2/3 \log N)(1 + o(1)),\label{e:L2} \quad \kappa > 0.
\end{align}
By the Law of Large Numbers, if the number of samples is sufficiently large for any fixed, but sufficiently large $N$, we can restate the result as
\begin{align*}
\mathbb P \left( \frac{T^{(1)}-\langle T^{(1)} \rangle}{\sigma_{T^{(1)}}} \leq t \right) \approx F^{\mathrm{gap}}_\beta( \kappa t + \mathbb E[\xi]).
\end{align*}
This is a universality theorem for the halting time $T^{(1)}$ as the limiting distribution does not depend on the distribution of the individual entries of the matrix ensemble, just whether it is real or complex. 

\begin{remark}
  If one constructs matrices $H = U \Lambda U^*$, $\Lambda = \diag(\lambda_N, \lambda_{N-1}, \ldots, \lambda_1)$ where the joint distribution of $\lambda_1 \leq \lambda_2 \leq \cdots \leq\lambda_N$ is given by
  \begin{align*}
    \propto \prod_{j=1}^N\E^{-N\frac{\beta}{2} V(\lambda_j)} \prod_{j< n} |\lambda_j-\lambda_n|^\beta,
  \end{align*}
  and $U$ is distributed (independently) according to Haar measure on either the orthogonal or unitary group then Theorem~\ref{t:main} holds for any $\beta \geq 1$.  Here $V$ should satisfy the hypotheses in Definition~\ref{def:IE}.
\end{remark}

\begin{remark}\label{r:other}
  To compute the largest eigenvalue of $H$, one can alternatively consider the flow
  \begin{align*}
    \dot X(t) = HX(t), \quad X(0) = [1,0,\ldots,0]^T.
  \end{align*}
  It follows that
  \begin{align*}
    \log \frac{\|X(t+1)\|}{\|X(t)\|} \to \lambda_N, \quad t \to \infty.
  \end{align*}
  So, define $T_{\mathrm{ODE}}(H) = \inf\left\{t: \left|\log \frac{\|X(t+1)\|}{\|X(t)\|} - \lambda_N  \right| \leq \epsilon \right\}$.  Using the proof technique we present here, one can show that Theorem~\ref{t:main} also holds with $T^{(1)}$ replaced with $T_{\mathrm{ODE}}$. \blue{The same is true for the power method, the inverse power method, and the QR algorithm without shifts on positive-definite random matrices (see \cite{Deift2017}).}
\end{remark}

\subsection{A numerical demonstration}\label{s:numerics}
We can demonstrate Theorem~\ref{t:main} numerically using the following WEs defined by letting $X_{ij}$ for $i \leq j$ be iid with distributions:
\begin{enumerate}
\item[\underline{GUE}] Mean zero standard complex normal.
\item[\underline{BUE}] $\xi + \I \eta$ where $\xi$ and $\eta$ are each the sum of independent mean zero Bernoulli random variables, \emph{i.e.} binomial random variables.
\item[\underline{GOE}] Mean zero standard (real) normal.
\item[\underline{BOE}] Mean zero Bernoulli random variable
\end{enumerate}
In Figure~\ref{f:TodaOne}, for $\beta = 2$,  we show how the histogram of $T^{(1)}$ (more precisely, $\tilde T^{(1)}$, see \eqref{e:ttilde} below), after rescaling, matches the density $\D/\D t F_2^{\mathrm{gap}}(t)$ which was computed numerically\footnote{Technically, the distribution of the first gap was computed , and then $F_2^\mathrm{gap}$ can be computed by a change of variables.  We thank Folkmar Bornemann for the data to plot $F_2^\mathrm{gap}$.} in \cite{Witte2013}. In Figure~\ref{f:TodaOneO}, for $\beta =1$, we show the histogram for $T^{(1)}$ (again, $\tilde T^{(1)}$), after rescaling, matches the density $\D/\D t F_1^{\mathrm{gap}}(t)$. To the best of our knowledge, a computationally viable formula for $\D/\D t F_1^{\mathrm{gap}}(t)$, analogous to $\D/\D t F_2^{\mathrm{gap}}(t)$ in \cite{Witte2013}, is not yet known and so we estimate the density $\D/\D t F_1^{\mathrm{gap}}(t)$ using Monte Carlo simulations with $N$ large. For convenience, we choose the variance for the above ensembles so that $[a_V,b_V]=[-2\sqrt{2},2\sqrt{2}]$ which, in turn, implies $c_{V} = 2^{-3/2}$.

It is clear from the proof of Theorem~\ref{t:main} that the convergence of the left-hand side in \eqref{e:main-thm} to $F_\beta^{\mathrm{gap}}$ is slow.  In fact, we expect a rate proportional to $1/\log N$.  This means that in order to demonstrate \eqref{e:main-thm} numerically with convincing accuracy one would have to consider very large values of $N$.  In order to display the convergence in \eqref{e:main-thm} for more reasonable values of $N$, we observe, using a simple calculation, that for any fixed $\gamma \neq 0$ the limiting distribution of
\begin{align}\label{e:ttilde}
 \tilde T^{(1)} = \tilde T_\gamma^{(1)} :=\frac{T^{(1)}}{c_V^{2/3} 2^{-2/3}N^{2/3}(\log \epsilon^{-1}-2/3 \log N + \gamma)}
\end{align}
as $N \to \infty$ is the same as for $\gamma = 0$.  A ``good" choice for $\gamma$ is obtained in the following way.  To analyze the $T^{(1)}$ in Sections~\ref{s:universality} and \ref{s:prob} below we utilize two approximations to $T^{(1)}$, viz. $T^*$ in \eqref{e:t-star} and $\hat T$ in \eqref{e:t-hat}:
\begin{align*}
T^{(1)} = \hat T + (T^{(1)} - T^*) + (T^* - \hat T).
\end{align*}
The parameter $\gamma$ can be inserted into the calculation by replacing $\hat T$ with  $\hat T_\gamma$
\begin{align*}
	\hat T \to \hat T_\gamma := \frac{(\alpha - 4/3) \log N + 2\gamma}{\delta_{N-1}}
\end{align*}
where $\gamma$ is chosen to make
\begin{align}\label{e:gamma-diff}
T^* - \hat T_\gamma = \frac{\log N^{2/3}(\lambda_N-\lambda_{N-1}) + \half \log \nu_{N-1} -\gamma}{\lambda_N - \lambda_{N-1}},
\end{align}
as small as possible.  Here $\nu_{N-1}$ and $\delta_{N-1}$ are defined at the beginning of Section~\ref{s:tech}.  Replacing $\log N^{2/3}(\lambda_N - \lambda_{N-1})$ and $\log \nu_N$ in \eqref{e:gamma-diff} with the expectation of their respective limiting distributions as $N \to \infty$ (see Theorem~\ref{t:gap-limit}: note that $\nu_{N-1}$ is asymptotically distributed as $\zeta^2$ where $\zeta$ is Cauchy distributed) we choose $\gamma_2= -\mathbb E(\log (c_V^{2/3}2^{-5/3}\xi_2)) + \half \mathbb E[ \log |\zeta|] \approx 0.883$ when $\beta = 2$ and $\gamma_1= -\mathbb E(\log (c_V^{2/3}2^{-5/3}\xi_1)) + \half \mathbb E[ \log |\zeta|] \approx 0.89$ when $\beta = 1$. Figures~\ref{f:TodaOne} and \ref{f:TodaOneO} are plotted using $\gamma_1$ and $\gamma_2$, respectively.


We can also examine the growth of the mean and standard deviation.  We see from Table~\ref{t:table} using a million samples and $\epsilon = 10^{-5}$, that the sample standard deviation is on the same order as the sample mean:
\begin{align}\label{e:order}
\sigma_{T^{(1)}} \sim \langle T^{(1)} \rangle \sim N^{2/3}(\log \epsilon^{-1}-2/3 \log N).
\end{align}
\begin{table}[ht]
\begin{align*}
\begin{array}{ccccccc}
N & 50 & 100 & 150 & 200 & 250 & 300 \\
\log \epsilon^{-1}/\log N-5/3 &1.28 & 0.833 & 0.631 & 0.506 & 0.418 & 0.352 \\
  \langle T^{(1)}\rangle\sigma_{T^{(1)}}^{-1} \text{ for GUE} &1.58 & 1.62 & 1.59 & 1.63 & 1.6 & 1.58 \\
 \langle T^{(1)}\rangle\sigma_{T^{(1)}}^{-1} \text{ for BUE} & 1.6 & 1.57 & 1.6 & 1.62 & 1.62 & 1.58 \\
 \langle T^{(1)}\rangle\sigma_{T^{(1)}}^{-1} \text{ for GOE} & 0.506 & 0.701 & 0.612 & 0.475 & 0.705 & 0.619 \\
  \langle T^{(1)}\rangle\sigma_{T^{(1)}}^{-1} \text{ for BOE} &0.717 & 0.649 & 0.663 & 0.747 & 0.63 & 0.708 \\
\end{array}
\end{align*}
\caption{\label{t:table} A numerical demonstration of \eqref{e:order}.  The second row of the table confirms that $(\epsilon,N)$ is in the scaling region for, say, $\sigma = 1/2$.  The last four rows demonstrate that the ratio of the sample mean to the sample standard deviation is order one. }
\end{table}

\begin{remark}
The ideas that allow us to establish \eqref{e:L1} for IEs requires the convergence of
\begin{align}\label{e:overgap}
\mathbb E \left[\frac{1}{N^{2/3}(\lambda_N - \lambda_{N-1})} \right].
\end{align}
For BUE, \eqref{e:overgap} must be infinite for all $N$ as there is a non-zero probability that the top two eigenvalues coincide owing to the fact that the matrix entries are discrete random variables.  Nevertheless, the sample mean and sample standard deviation of $T^{(1)}$ are observed to converge, after rescaling.  It is an interesting open problem to show that  convergence in \eqref{e:L1} still holds in this case of discrete WEs even though \eqref{e:overgap} is infinite.  Specifically, the convergence in the definition of $\xi$ (Definition~\ref{def:Fbeta}) for discrete WEs cannot take place in expectation. Hence $T^{(1)}$ acts as a molified version of the inverse of the top gap --- it is always finite.
\end{remark}

\subsection{Estimates from random matrix theory}\label{s:rmt}

We now introduce the results from random matrix theory that are needed to prove Theorem~\ref{t:main} and Proposition~\ref{p:error}.  Let $H$ be an $N \times N$ Hermitian (or just real symmetric) matrix with eigenvalues $\lambda_1 \leq \lambda_2 \leq \cdots \leq\lambda_N$ and let $\beta_1, \beta_2, \ldots, \beta_{N}$ denote the absolute value of the first components of the normalized eigenvectors.  \blue{We assume the entries of $H$ are distributed according to an invariant or generalized Wigner ensemble (see Appendix~\ref{s:ensemble}).}   Define the averaged empirical spectral measure
\begin{align*}
\mu_N(z) = \mathbb E \frac{1}{N} \sum_{i=1}^N \delta(\lambda_i-z),
\end{align*}
where the expectation is taken with respect to the given ensemble.

\begin{theorem}[Equilibrium measure, \cite{Bourgade2014}]\label{def:EM}
For any WE or IE the measure $\mu_N$ converges weakly to a measure $\mu$, called the equilibrium measure, which has support on a single interval $[a_V,b_V]$ and, for suitable constants $C_\mu$ and $c_V$, has a density $\rho$ that satisfies $\rho(x) \leq C_\mu \sqrt{b_V-x}\chi_{(-\infty,b_V]}(x)$ and $\rho(x) = \frac{2^{3/4}c_V}{\pi} \sqrt{b_V-x}(1 + \bigo(b_V-x))$ as $x \to b_V$.
\end{theorem}

With the chosen normalization for WEs, $\sum_{i=1}^N \sigma_{ij}^2 = 1$, $[a_V,b_V] = [-2,2]$ and $c_V = 1$ \cite{Bourgade2014}.  One can vary the support as desired by shifting and scaling, $H \to aH + bI$: the constant $c_V$ then changes accordingly. When the entries of $H$ are distributed according to a WE or an IE with high probability (see Theorem~\ref{t:generic}) the top three eigenvalues are distinct and $\beta_j \neq 0$ for $j=N,N-1,N-2$.  Next, let $d\mu$ denote the limiting spectral density or equilibrium measure for the ensemble as $N \to \infty$.  Then define $\gamma_n$ to be the smallest value of $t$ such that
\begin{align*}
\frac{n}{N} = \int_{-\infty}^t \D\mu.
\end{align*}
Thus $\{\gamma_n\}$ represent the quantiles of the equilibrium measure.


There are four fundamental parameters involved in our calculations.  First we fix $0 < \sigma < 1$ once and for all, then we fix $0 < p < 1/3$, then we choose $s < \min\{\sigma/44, p/8\}$ and then finally $0 < c \leq 10/\sigma$ will be a constant that will allow us to estimate the size of various sums.  The specific meanings of the first three parameters is given below.  Also, $C$ denotes a generic constant that can depend on $\sigma$ or $p$ but not on $s$ or $N$.  We also make statements that will be ``true for $N$ sufficiently large".  This should be taken to mean that there exists $N^* = N^*(\mu,\sigma,s,p)$ such that the statement is true for $N > N^*$. For convenience in what follows we use the notation $\epsilon = N^{-\alpha/2}$, so
\begin{center}
$(\epsilon,N)$ are in the scaling region if and only if $\alpha -10/3 \geq \sigma > 0$
\end{center}
and $\alpha = \alpha_N$ is allowed to vary with $N$.  \blue{Our calculations that follow involve first deterministic estimates and then probabilistic estimates. The following conditions provide the setting for the deterministic estimates.}

\begin{condition} \label{cond:gap} For $0 < p < \sigma/4$,
\begin{itemize}
 \item $\lambda_{N-1}-\lambda_{N-2} \geq p(\lambda_N - \lambda_{N-1})$.
\end{itemize}
Let $G_{N,p}$ denote the set of matrices that satisfy this condition.
\end{condition}

\begin{condition}\label{cond:rigidity}
For any fixed $0<  s < \min\{\sigma/44, p/8\}$
\begin{enumerate}
\item $\beta_n \leq N^{-1/2+s/2}$ for all $n$
\item $N^{-1/2-s/2} \leq \beta_n$ for $n = N, N-1$,
\item $N^{-2/3-s} \leq \lambda_N - \lambda_{n-1} \leq N^{-2/3+s}$, for $n = N, N-1$, and
\item $|\lambda_n - \gamma_n| \leq N^{-2/3+s}( \min\{n,N-n+1\})^{-1/3}$ for all $n$.
\end{enumerate}
Let $R_{N,s}$ denote the set of matrices that satisfy these conditions.
\end{condition}

\begin{remark}\label{r:IE}
It is known that the distribution (Haar measure on the unitary or orthogonal group) on the eigenvectors for IEs depends only on $\beta = 1,2$.  And, if $V(x) = x^2$ the IE is also a WE.  Therefore, if one can prove a general statement about the eigenvectors for WEs then it must also hold for IEs.  But, it should be noted that stronger results can be proved for the eigenvectors for IEs, see \cite{Stam1982} and \cite{Jiang2006} for example.
\end{remark}

The following theorem has it roots in the pursuit of proving universality in random matrix theory.  See \cite{TracyWidom} for the seminal result when $V(x) = x^2$ and $\beta = 2 	$.  Further extensions include the works of Soshnikov \cite{Soshnikov1999} and Tao and Vu \cite{Tao2010} for Wigner ensembles and \cite{Deift2007a} for invariant ensembles.
\begin{theorem}\label{t:gap-limit}
For both IEs and WEs
\begin{align*}
N^{1/2}(|\beta_N|, |\beta_{N-1}|, |\beta_{N-2}|)
\end{align*}
converges jointly in distribution to $(|X_1|,|X_2|,|X_3|)$ where $\{X_1,X_2,X_3\}$ are iid real ($\beta = 1$) or complex ($\beta =2$) standard normal random variables.  Additionally, for IEs and WEs
\begin{align*}
2^{-2/3}N^{2/3}(b_V - \lambda_N, b_V - \lambda_{N-1}, b_V-\lambda_{N-2})
\end{align*}
converges jointly in distribution to random variables $(\Lambda_{1,\beta},\Lambda_{2,\beta},\Lambda_{3,\beta})$ which are the smallest three eigenvalues of the so-called stochastic Airy operator. Furthermore, $(\Lambda_{1,\beta},\Lambda_{2,\beta},\Lambda_{3,\beta})$ are distinct with probability one.
\end{theorem}
\begin{proof}
The first claim follows from \cite[Theorem 1.2]{Bourgade2013}.  The second claim follows from \cite[Corollary 2.2 \& Theorem 2.7]{Bourgade2014}.  The last claim follows from \cite[Theorem 1.1]{Ramirez2011}.
\end{proof}

\begin{definition}\label{def:Fbeta}
The distribution function $F^\mathrm{gap}_\beta(t)$ for $\beta =1,2$ is given by
\begin{align*}
F^\mathrm{gap}_\beta(t) = \mathbb P\left( \frac{1}{\Lambda_{2,\beta}-\Lambda_{1,\beta}} \leq t \right) = \lim_{N \to \infty} \mathbb P\left( \frac{1}{c_V^{2/3} 2^{-2/3}N^{2/3}(\lambda_N-\lambda_{N-1})} \leq t \right), \quad t \geq 0.
\end{align*}
\end{definition}
\noindent Properties of $G_\beta(t) := 1-F_\beta^{\mathrm{gap}}(1/t)$, the distribution function for the first gap, are examined in \cite{Perret2013,Witte2013,Monthus2013} including the behavior of $G_\beta(t)$ near $ t= 0$ which is critical for understanding which moments of $F_\beta'(t)$ exist.

The remaining theorems in this section are compiled from results that have been obtained recently in the literature.  \blue{These results show that the conditions described above hold with arbitrarily high probability.}
\begin{theorem}\label{t:generic}
For WEs or IEs Condition~\ref{cond:rigidity} holds with high probability as $N \to \infty$, that is, for any $s > 0$
\begin{align*}
\mathbb P(R_{N,s}) = 1 + o(1),
\end{align*}
as $N \to \infty$.
\end{theorem}
\begin{proof}
We first consider WEs.  The fact that the probability of \cond{rigidity}.1 tends to unity follows from \cite[Theorem 2.1]{Erdos2012a} using estimates on the (1,1) entry of the Green's function.  See \cite[Section 2.1]{Erdos2012b} for a discussion of using these estimates.  The fact that the probability of each of \cond{rigidity}.2-3  tends to unity follows from Theorem~\ref{t:gap-limit} using Corollary~\ref{c:whp}.  And finally, the statement that the probability \cond{rigidity}.4 tends to unity as $N \to \infty$ is the statement of the rigidity of eigenvalues, the main result of \cite{Erdos2012a}.  Following Remark~\ref{r:IE}, we then have that the probability of \cond{rigidity}.1-2 tends to unity for IEs.

For IEs, the fact that the probability of \cond{rigidity}.4 tends to unity follows from \cite[Theorem 2.4]{Bourgade2013}.  Again, the fact that the probability of \cond{rigidity}.3  tends to unity follows from Theorem~\ref{t:gap-limit} using Corollary~\ref{l:whp}.
\end{proof}

\begin{theorem}\label{t:p}
For both WEs and IEs
\begin{align*}
\lim_{p \downarrow 0}\limsup_{N \to \infty} \mathbb P(G_{N,p}^c) = 0.
\end{align*}
\end{theorem}
\begin{proof}
It follows from Theorem~\ref{t:gap-limit} that
\begin{align*}
\limsup_{N \to \infty} \mathbb P(G_{N,p}^c) = \lim_{N\to \infty} \mathbb P(\lambda_{N-1}-\lambda_{N-2} < p(\lambda_N-\lambda_{N-1}))
&= \mathbb P(\Lambda_{3,\beta} -\Lambda_{2,\beta} < p(\Lambda_{2,\beta} -\Lambda_{1,\beta})).
\end{align*}
Then
\begin{align*}
\lim_{p \downarrow 0} \mathbb P(\Lambda_{3,\beta} -\Lambda_{2,\beta} < p(\Lambda_{2,\beta} -\Lambda_{1,\beta})) = \mathbb P \left( \bigcap_{p > 0}\left\{\Lambda_{3,\beta} -\Lambda_{2,\beta} < p(\Lambda_{2,\beta} -\Lambda_{1,\beta}) \right\} \right) = \mathbb P(\Lambda_{3,\beta} = \Lambda_{2,\beta}).
\end{align*}
But from \cite[Theorem 1.1]{Ramirez2011} $\mathbb P(\Lambda_{3,\beta} = \Lambda_{2,\beta}) = 0$.
\end{proof}

\noindent Throughout what follows we assume we are given a WE or an IE.


\subsection{Technical lemmas}\label{s:tech}

Define $\delta_j = 2(\lambda_{N}-\lambda_j)$ and $I_c = \{1 \leq n \leq N-1: \delta_n/\delta_{N-1} \geq 1 + c \}$ for $c > 0$.
\begin{lemma}\label{l:Ic}
  Let $0 < c < 10/\sigma$.  Given Condition~\ref{cond:rigidity}
  \begin{align*}
    |I_c^c| \leq N^{2s}
  \end{align*}
  for $N$ sufficiently large, where $^c$ denotes the compliment relative to $\{1,\ldots,N-1\}$.
\end{lemma}
\begin{proof}
  We use rigidity of the eigenvalues, Condition~\ref{cond:rigidity}.4. So, $|\lambda_n - \gamma_n| \leq N^{-2/3+s} (\hat n)^{-1/3}$ where $\hat n = \min\{n, N-n + 1\}$.  Recall
  \begin{align*}
    I_c^c \subset \{1 \leq n \leq N-1: \lambda_{N}-\lambda_n < (1+c)(\lambda_{N}-\lambda_{N-1})\}.
  \end{align*}
  Define
  \begin{align*}
    J_c = \{1 \leq n \leq N-1: \gamma_N - \gamma_n  \leq (2 + c + (\hat n)^{-1/3}) N^{-2/3+s}\}.
  \end{align*}
  If $n \in I_c$ then
  \begin{align*}
    \lambda_N - \lambda_n &\leq (1+c) N^{-2/3+s},\\
    \gamma_N-N^{-2/3+s} - (\gamma + (\hat n)^{-1/3} N^{-2/3+s}) & \leq \lambda_N - \lambda_n \leq (1+c) N^{-2/3+s},\\
    \gamma_N-\gamma_n &\leq (2 + c + (\hat n)^{-1/3}) N^{-2/3+s},
  \end{align*}
  and hence $n \in J_c$.  Then compute the  asymptotic size of the set $J_c$ let $n^*$ be the smallest element of $J_c$.  Then $|J^*| = N - n^*$ so that
  \begin{align*}
  \frac{n^*}{N} = \int_{-\infty}^{\gamma_{n^*}} \D\mu, \quad |I_c^c| \leq |J_c| =   N - n^* = N\int_{n^*}^\infty \D \mu.
  \end{align*}
  Then using Definition~\ref{def:EM}, $\gamma_N = b_V$ and $n^* \geq b_V - (2+c+ (\hat n)^{-1/3} )N^{-2/3+s} \geq b_V - (3+c)N^{-2/3+s}$ to see
  \begin{align*}
  |I_c^c| \leq N\int_{n^*}^\infty \D \mu \leq C_\mu N \int_{b_V - (3+c)N^{-2/3+s}}^{b_V} \sqrt{b_V -x}\,\D x = \frac{2C_\mu}{3} (3+c)^{3/2}N^{3s/2}.
  \end{align*}
  and then because $\sigma$ is fixed hence $c$ has an upper bound and $s > 0$, $|I_c^c| \leq N^{2s}$ for sufficiently large $N$.
  \end{proof}

    \noindent We use the notation $\nu_n = \beta_n^2/\beta_N^2$ and note that for a matrix in $R_{N,s}$ we have $\nu_n \leq N^{2s}$ and $\sum_n \nu_n = \beta_N^{-2} \leq N^{1+s}$.  One of the main tasks that will follow is estimating the following sums.
  \begin{lemma}\label{l:estimate}
	Given Condition~\ref{cond:rigidity}, $0 < c \leq 10/\sigma$ and $j \leq 3$ there exists an absolute constant $C$ such that
    \begin{align*}
    	 N^{-2s} \delta_{N-1}^j \E^{-\delta_{N-1}t} \leq \sum_{n=1}^{N-1} \nu_n \delta^j_n \E^{-\delta_n t} \leq C \E^{-\delta_{N-1}t} \left(N^{4s} \delta_{N-1}^j + N^{1+s} \E^{-c \delta_{N-1} t} \right),
	\end{align*}
    for $N$ sufficiently large.
\end{lemma}
	\begin{proof}
    For $j \leq 3$
    \begin{align*}
    \sum_{n=1}^{N-1} \nu_n \delta^j_n \E^{-\delta_n t} &= \left( \sum_{n \in I_c} + \sum_{n \in I_c^c}\right) \nu_n \delta^j_n \E^{-\delta_n t} \\
    &\leq \sum_{n \in I^c_c} \nu_n (1+c)^j \delta_{N-1}^j \E^{-\delta_{N-1} t} + 2^j\sum_{n \in I_c} \nu_n |\lambda_1-\lambda_N|^j \E^{- (1+c) \delta_{N-1} t}.
    \end{align*}
    It also follows that $\lambda_N-\lambda_1 \leq b_V-a_V +1$ so that by Lemma~\ref{l:Ic} for sufficiently large $N$
    \begin{align*}
   \sum_{n=1}^{N-1} \nu_n \delta^j_n \E^{-\delta_n t} &\leq C \E^{-\delta_{N-1}t} \left(N^{4s} \delta_{N-1}^j + N^{1+s} \E^{-c \delta_{N-1} t} \right).
	\end{align*}
    To find a lower bound, we just keep the first term, as that should be the largest
    \begin{align*}
    \sum_{n=1}^{N-1} \nu_n \delta^j_n \E^{-\delta_n t} \geq \nu_{N-1} \delta^j_{N-1} \E^{-\delta_{N-1} t} \geq N^{-2s} \delta_{N-1}^j \E^{-\delta_{N-1}t}.
	\end{align*}
    \end{proof}

\section{Estimates for the Toda algorithm}\label{s:universality}
Remarkably, \eqref{e:Toda} can be solved explicitly by a QR factorization procedure, see for example \cite{Symes1982}. For $X(0) = H$ we have for $t \geq 0$
\begin{align*}
\E^{tH} = Q(t) R(t),
\end{align*}
where $Q$ is orthogonal ($\beta = 1$) or unitary ($\beta = 2$) and $R$ has positive diagonal entries.  This \emph{QR factorization} for $\E^{tH}$ is unique: Note that $Q(t)$ is obtained by apply Gram--Schmidt to the columns of $\E^{tH}$. We claim that $X(t) = Q^*(t) H Q(t)$ is the solution of \eqref{e:Toda}.  Indeed, by differentiating, we obtain
\begin{align}
H \E^{tH} &= H Q(t)R(t) = \dot Q(t)R(t) + Q(t)\dot R(t),\notag\\
X(t) &= Q^*(t) \dot Q(t) + \dot R(t) R^{-1}(t).\label{e:X}
\end{align}
Then because $\dot R(t) R^{-1}(t)$ is upper triangular
\begin{align*}
(X(t))_- = (Q^*(t)\dot Q(t))_-.
\end{align*}
Furthermore, from $Q^*(t)Q(t) = I$ we have $Q^*(t) \dot Q(t) = - \dot Q^*(t) Q(t)$ so that $Q^*(t)\dot Q(t)$ is skew Hermitian.  Thus, $B(X(t)) = Q^*(t)\dot Q(t)-[Q^*(t)\dot Q(t)]_D$, where $[\cdot]_D$ gives the diagonal part of the matrix.  However, as $Q^*(t) \dot Q(t)$ is skew Hermitian $[Q^*(t)\dot Q(t)]_D$ is purely imaginary.  On the other hand, we see from \eqref{e:X} that the diagonal is real.  It follows that $[Q^*(t)\dot Q(t)]_D = 0$ and $B(X(t)) = Q^*(t)\dot Q(t)$. Using \eqref{e:Toda} we have
\begin{align*}
\dot X(t) & = \dot Q^*(t) H Q(t) + Q^*(t) H \dot Q(t),
\end{align*}
and so
\begin{align*}
\dot X(t) = X(t)B(X(t)) &- B(X(t)) X(t).
\end{align*}
When $t = 0$, $Q(0) = I$ so that $X(0) = H$ and by uniqueness for ODEs this shows $X(t)$ is indeed the solution of \eqref{e:Toda}.

As the eigenvalues of $X(0) = H$ are not necessarily simple (indeed for BOE there is a non-zero probability for a matrix to have repeated eigenvalues), it is not clear a priori that the eigenvectors of $X(t)$ can be chosen to be smooth functions of $t$.  However, for the case at hand we can proceed in the following way.  For $X(0)= H$ there exists a (not necessarily unique) unitary matrix $U_0$ such that $X(0) = U_0\Lambda U_0^*$ where $\Lambda = \diag(\lambda_1,\ldots,\lambda_N)$.  Then $X(t) = Q^*(t)HQ(t) = U(t) \Lambda U^*(t)$ where $U(t) = Q^*(t) U_0$.  Then the $j$th column $u_j(t)$ of $U(t)$ is a smooth eigenvector of $X(t)$ corresponding to eigenvalue $\lambda_{j}$.   From the eigenvalue equation
\begin{align*}
(X(t)- \lambda_j) u_j(t) &= 0,
\end{align*}
we obtain (following Moser \cite{Moser1975})
\begin{align*}
\dot X(t)u_j(t) &+ (X(t)- \lambda_j) \dot u_j(t) = 0,\\
(X(t) B(X(t)) &- B(X(t))X(t)) u_j(t) + (X(t)-\lambda_j) \dot u_j(t) = 0,\\
(X(t) &- \lambda_j)[\dot u_j(t) + B(X(t)) u_j(t)] = 0.
\end{align*}
This last equation implies  $\dot u_j(t) + B(X(t))u_j$ must be a (possibly time-dependent) linear combination of the eigenvectors corresponding to $\lambda_j$.  Let $U_j(t) = [u_{j_1}(t),\ldots,u_{j_m}(t)]$ be eigenvectors corresponding to a repeated eigenvalue $\lambda_j$ so that for $i = 1,\ldots,m$
\begin{align*}
\dot u_{j_i}(t) + B(X(t)) u_{j_i}(t) &= \sum_{k=1}^m d_{ki}(t) u_{j_k}(t),
\end{align*}
and so
\begin{align}\label{e:Uj}
\left[\frac{d}{dt} + B(X(t)) \right] U_j(t) &= U_j(t) D(t), \quad D(t) = (d_{ki}(t))_{k,i=1}^m.
\end{align}
Note that $U_j^*(t) U_j(t) = I_m$, the $m\times m$ identity matrix. Then multiplying \eqref{e:Uj} on the left by $U_j^*(t)$ and then multiplying the conjugate transpose of \eqref{e:Uj} on the right by $U_j(t)$, we obtain
\begin{align*}
U_j^*(t)\dot U_j(t) + U_j^*(t) B(X(t)) U_j(t) &= D(t),\\
\dot U_j^*(t) U_j(t) + U_j^*(t) [B(X(t))]^* U_j(t) &= D^*(t).
\end{align*}
Because $\D/\D t [U_j^*(t)U_j(t)] = 0$ and $B(X(t))$ is skew Hermitian, the addition of these two equations gives $D(t) = - D^*(t)$.  Let $S(t)$ be the solution of $\dot S(t) = - D(t)S(t)$ with $S(0) = I_m$.  Then $\D/\D t [S^*(t)S(t)] = -S^*(t)D(t)S(t) + S^*(t)D(t)S(t) = 0$ and hence $S^*(t)S(t) = C = I_m$, \emph{i.e.}, $S(t)$ is unitary. In particular, $\tilde U_j(t) :=  U_j(t) S(t)$ has orthonormal columns and we find
\begin{align*}
\left[\frac{d}{dt} + B(X(t)) \right] \tilde U_j(t) = U_j(t) D(t)S(t) - U_j(t) D(t) S(s) = 0.
\end{align*}
We see that a smooth normalization for the eigenvectors of $X(t)$ can always be chosen so that $D(t)=0$.  Without loss of generality, we can assume that $U(t)$ solves \eqref{e:Uj} with $D(t) = 0$. Then for $U(t) = (U_{ij}(t))_{i,j=1}^N$
\begin{align*}
\dot U_{1j}(t) & = - e_1^* B(X(t))u_j(t) = (B(X(T)) e_1)^* u_j(t) = (X(t)e_1 - X_{11}^*(t)e_1)^*u_j(t)\\
& = e_1^*(X(t) - X_{11}(t))u_j(t)=(\lambda_j - X_{11}(t)) U_{1j}(t).
\end{align*}
A direct calculation using
\begin{align*}
X_{11}(t) = e_1^* X(t)e_1 = \sum_{j=1}^N \lambda_j |U_{1j}(t)|^2,
\end{align*}
shows that
\begin{align*}
U_{1j}(t) = \frac{U_{1j}(0) \E^{\lambda_j t}}{\ds \left(\sum_{j=1}^N |U_{1j}(0)|^2 \E^{2 \lambda_j t} \right)^{1/2}}, \quad 1 \leq j \leq N.
\end{align*}
Also
\begin{align*}
X_{1k}(t) =\sum_{j=1}^N \lambda_j U_{1j}^*(t) U_{kj}(t),
\end{align*}
and hence
\begin{align*}
\sum_{k=2}^N |X_{1k}(t)|^2 &= \sum_{k=2}^N X_{1k}(t)X_{k1}(t) = [X^2(t)]_{11}-X^2_{11}(t)\\
&= \sum_{k=1}^N \lambda_k^2 |U_{1k}(t)|^2 - \left( \sum_{k=1}^N \lambda_k |U_{1k}(t)|^2\right)^2 = \sum_{k=1}^N (\lambda_k - X_{11}(t))^2 |U_{1k}(t)|^2.
\end{align*}
Thus
\begin{align*}
E(t) :=\sum_{k=2}^N |X_{1k}(t)|^2 =  \sum_{j=1}^N (\lambda_j-X_{11}(t))^2 |U_{1j}(t)|^2.
\end{align*}
  We also note that
\begin{align*}
\lambda_N -X_{11}(t) = \sum_{j=1}^N (\lambda_N-\lambda_j) |U_{1j}(t)|^2.
\end{align*}
From these calculations, if $U_{11}(0) \neq 0$, it follows that
\begin{align*}
\begin{array}{c}
X_{11}(t) - \lambda_{N}\\
E(t)
\end{array} \to 0, \quad N \to \infty.
\end{align*}
While $X_{11}(t)-\lambda_{N}$ is of course the true error in computing $\lambda_N$ we use $E(t)$ to determine a convergence criterion as it is easily observable:  Indeed, as noted above, if  $E(t) < \epsilon$ then $|X_{11}(t) -\lambda_j|< \epsilon$, for some $j$. With high probability, $\lambda_j = \lambda_N$.

Note that, in particular, from the above formulae, $E(t)$ and $\lambda_N-X_{11}(t)$ depend \textbf{only} on the eigenvalues and the moduli of the first components of the eigenvectors of $X(0) = H$.  This fact is critical to our analysis.  With the notation $\beta_j = |U_{1j}(0)|$ we have that
\begin{align*}
|U_{1j}(t)| = \frac{\beta_j \E^{\lambda_jt}}{\ds \left(\sum_{n=1}^N \beta_n^2 \E^{2\lambda_n t} \right)^{1/2}}.
\end{align*}
A direct calculation shows that
\begin{align*}
E(t) &= E_0(t) + E_1(t),
\end{align*}
where
\begin{align*}
E_0(t) &= \frac{1}{4}\frac{\ds\sum_{n=1}^{N-1} \delta_{n}^2 \nu_n \E^{-\delta_{n} t}}{\ds\left(1+ \sum_{n=1}^{N-1} \nu_n \E^{-\delta_n t}\right)^2}.\\
E_1(t) &= \frac{\ds \left(\sum_{n=1}^{N-1} \lambda_n^2 \nu_n \E^{-\delta_{n} t}\right)\left(\sum_{n=1}^{N-1} \nu_n \E^{-\delta_n t}\right) - \left(\sum_{n=1}^{N-1} \lambda_n \nu_n \E^{-\delta_n t}\right)^2}{\ds\left(1+ \sum_{n=1}^{N-1} \nu_n \E^{-\delta_n t}\right)^2}.
\end{align*}
Note that $E_1(t) \geq 0$ by the Cauchy--Schwarz inequality; of course, $E_0(t)$ is trivially positive. It follows that $E(t)$ is small if and only if both $E_0(t)$ and $E_1(t)$ are small, a fact that is extremely useful in our analysis.

In terms of the probability $\rho_N$ measure on $\{1,2,\ldots,N\}$ defined by
\begin{align*}
\rho_{N}(E) = \left(\sum_{n=1}^N \nu_n \E^{-\delta_n t} \right)^{-1} \sum_{n \in E}\nu_n \E^{-\delta_n t},
\end{align*}
and a function $\lambda(j) = \lambda_j$
\begin{align*}
E(t) = \mathrm{Var}_{\rho_N}(\lambda).
\end{align*}
We will also use the alternate expression
\begin{align}\label{e:H1}
E_1(t) = \left( \frac{\ds\sum_{n=1}^{N-1} \nu_n \E^{-\delta_n t}}{\displaystyle 1+ \sum_{n=1}^{N-1} \nu_n \E^{-\delta_n t}}\right)^2 \mathrm{Var}_{\rho_{N-1}}(\lambda).
\end{align}
Additionally,
\begin{align}\label{e:true-error}
\lambda_N - X_{11}(t)   = \half\frac{\displaystyle \sum_{n=1}^N \delta_n\beta_n^2 \E^{2 \lambda_n t}}{\displaystyle 1+\sum_{n=1}^{N-1} \beta_n^2 \E^{2 \lambda_n t}}.
\end{align}

\subsection{The halting time and its approximation}

To aid the reader we provide a glossary to summarize inequalities for parameters and quantities that have previously appeared:
\begin{enumerate}
\item $0 < \sigma < 1$ is fixed,
\item $0 < p < 1/3$,
\item $\alpha \geq 10/3 + \sigma$,
\item $s \leq \min\{\sigma/44,p/8\}$,
\item $\alpha -4/3 - 44s \geq 2$,
\item $c \leq 10/\sigma$ can be chosen for convenience line by line when estimating sums with Lemma~\ref{l:estimate},
\item $\delta_n = 2(\lambda_N - \lambda_n)$,
\item $\nu_n = \beta_n^2/\beta_N^2$,
\item given Condition~\ref{cond:rigidity}
\begin{itemize}
\item $2 N^{-2/3 - s} \leq \delta_{N-1} \leq 2 N^{-2/3+s}$,
\item $N^{-2s} \leq \nu_n \leq N^{2s}$,
\item $\displaystyle \sum_{n=1}^{j} \nu_n \leq \sum_{n=1}^{N} \nu_n = \beta_N^{-2} \leq N^{1+s}$, for $1 \leq j \leq N$, and
\end{itemize}
\item $C > 0$ is a generic constant.
\end{enumerate}

\begin{definition}\label{def:halting}
The halting time (or the $1$-deflation time) for the Toda lattice (compare with \eqref{e:toda-halt}) is defined to be
\begin{align*}
T^{(1)} = \inf\{t: E(t)\leq \epsilon^2\}.
\end{align*}
\end{definition}

We find bounds on the halting time.

\begin{lemma}\label{l:toda-T}
Given Condition~\ref{cond:rigidity}, the halting time $T$ for the Toda lattice satisfies
\begin{align*}
(\alpha - 4/3 - 5s) \log N/\delta_{N-1} \leq T^{(1)} \leq (\alpha - 4/3 + 7s) \log N /\delta_{N-1},
\end{align*}
for sufficiently large $N$.
\end{lemma}
\begin{proof}
We use that $E(t) \geq E_0(t)$ so if $E_0(t) > N^{-\alpha}$ then $T^{(1)} \geq t$.  First, we show that $E_0(t) > \epsilon^2$, $0 \leq t \leq \sigma/2 \log N/\delta_{N-1}$ and sufficiently large $N$ and then we use this to show that $E_0(t) > \epsilon^2$, $t \leq (\alpha-4/3-5s) \log N/\delta_{N-1}$ and sufficiently large $N$.

Indeed, assume $t = a \log N/\delta_{N-1}$ for $0 \leq a \leq \sigma/2$.  Using Lemma~\ref{l:estimate}
\begin{align}\label{e:s1}
1 + \sum_{n=1}^{N-1} \nu_n \E^{-\delta_n t} \leq 1+ C \E^{-\delta_{N-1} t} \left( N^{4s} + N^{1+s} \E^{-c \delta_{N-1} t} \right).
\end{align}
Then using Lemma~\ref{l:estimate} we have
\begin{align*}
E_0(t) \geq  N^{-2s} \delta_{N-1}^2 \E^{-\delta_{N-1}t} \left( 1+ C \E^{-\delta_{N-1} t} \left( N^{4s} + N^{1+s} \E^{-c \delta_{N-1} t} \right) \right)^{-2}.
\end{align*}
Since $a \leq \sigma/2$ and we find
\begin{align*}
E_0(t) \geq N^{-4s-4/3-\sigma/2} \left(1 + C(N^{4s} + N^{1+s})\right)^{-2} \geq C N^{-8s-10/3-\sigma/2},
\end{align*}
for some new constant $C > 0$.  This last inequality follows because $N^{4s} \leq N^{1+s}$ as $s \leq 1/44$ (see Condition~\ref{cond:rigidity}). But then from Definition~\ref{def:scaling} this right-hand side is larger than $\epsilon^2 = N^{-\alpha}$ for sufficiently large $N$.
Now, assume $t = a \log N/\delta_{N-1}$ for $\sigma/2 \leq a \leq (\alpha - 4/3 - 5s) \log N/\delta{N-1}$.  We choose $c = 2(2+s)/\sigma \leq 10/\sigma$
\begin{align*}
E_0(t) &\geq \frac{1}{4} N^{-4s-4/3-a} \left(1 + C(N^{4s-a} + N^{1+s-ca})\right)^{-2} \\
&\geq N^{-\alpha +s}(1 + C(N^{4s-\sigma/2} + N^{-1} ) ) > N^{-\alpha}
\end{align*}
 for sufficiently large $N$.  Here we used that $s \leq \sigma/44$.  This shows
$(\alpha-4/3-5s) \log N/\delta_{N-1} \leq T^{(1)}$ for $N$ sufficiently large.

 Now, we work on the upper bound.   Let $t = a \log N/\delta_{N-1}$ for $a \geq (\alpha -4/3 + 7s)$ and we find using Lemma~\ref{l:estimate}
\begin{align*}
	E_0(t) \leq C N^{-a} \left( N^{-4/3 + 6s} + N^{1+s-ca} \right).
\end{align*}
Then using the minimum value for $a$
\begin{align*}
E_0(t) \leq N^{-\alpha}\left( C(N^{-s} + CN^{1+7s-ca+4/3}) \right).
\end{align*}
It follows from Definition~\ref{def:scaling} that $a \geq 10/3+\sigma-4/3 + 7s > 2$.  If we set $c = 2$ and use $s \leq 1/44$ then $1+7s-ca+4/3 \leq -3 + 4/3 + 7s \leq -2$
\begin{align*}
E_0(t) \leq N^{-\alpha}\left( C(N^{-s} + CN^{-2}) \right) < CN^{-\alpha-s}
\end{align*}
for sufficiently large $N$.

Next, we must estimate $E_1(t)$ when $a \geq (\alpha -4/3 + 7s)$.  We use \eqref{e:H1} and $\mathrm{Var}_{N-1}(\lambda) \leq C$.  Then by \eqref{e:s1}
\begin{align*}
E_1(t) \leq C N^{-2a}(N^{4s}+N^{1+s-ca})^2.
\end{align*}
Again, using $c = 1$ and the fact that $a > 2$ we have
\begin{align}\label{e:H1-est}
E_1(t) \leq C N^{-\alpha} N^{8s-\alpha + 8/3 -14s} \leq C N^{-\alpha} N^{-\alpha+8/3} \leq N^{-\alpha}
\end{align}
for $N$ sufficiently large.  This shows $T^{(1)} \leq (\alpha-4/3+7s)\log N/\delta_{N-1}$ for sufficiently large $N$ as $E(t) = E_0(t) + E_1(t) \leq \epsilon^2$ if $t < (\alpha-4/3+7s)\log N/\delta_{N-1}$ and $N$ is sufficiently large.
\end{proof}
In light of this lemma we define
\begin{align*}
I_\alpha = [(\alpha-4/3-5s) \log N/\delta_{N-1},(\alpha-4/3+7s) \log N/\delta_{N-1}].
\end{align*}

Next, we estimate the derivative of $E_0(t)$.  We find
\begin{align}\label{e:DH0-form}
E_0'(t) = \frac{\ds -\left(\sum_{n=1}^{N-1} \delta_{n}^3 \nu_n \E^{-\delta_{n} t}\right)\left(1+ \sum_{n=1}^{N-1} \nu_n \E^{-\delta_n t}\right) + 2\left(\sum_{n=1}^{N-1} \delta_{n}^2 \nu_n \E^{-\delta_{n} t}\right) \left(\sum_{n=1}^{N-1} \delta_n \nu_n \E^{-\delta_n t}\right)}{\ds\left(1+ \sum_{n=1}^{N-1} \nu_n \E^{-\delta_n t}\right)^3}.
\end{align}
\begin{lemma}\label{l:DH0}
Given Condition~\ref{cond:rigidity} and $t \in I_\alpha$
\begin{align*}
-E_0'(t) \geq C N^{-12s-\alpha-2/3},
\end{align*}
for sufficiently large $N$.
\end{lemma}
\begin{proof}
We use \eqref{e:DH0-form}.  The denominator is bounded below by unity so we estimate the numerator. By Lemma~\ref{l:estimate}
\begin{align*}
\left(\sum_{n=1}^{N-1} \delta_{n}^3 \nu_n \E^{-\delta_{n} t}\right)\left(1+ \sum_{n=1}^{N-1} \nu_n \E^{-\delta_n t}\right) \geq \sum_{n=1}^{N-1} \delta_{n}^3 \nu_n \E^{-\delta_{n} t} \geq N^{-2s} \delta^3_{N-1} \E^{-\delta_{N-1}t}.
\end{align*}
For $t \in I_\alpha$
\begin{align*}
N^{-2s} \delta^3_{N-1} \E^{-\delta_{N-1}t} \geq N^{-12s-2/3-\alpha}.
\end{align*}
Next, again by Lemma~\ref{l:estimate}
\begin{align*}
\left(\sum_{n=1}^{N-1} \delta_{n} \nu_n \E^{-\delta_{n} t}\right)\left(\sum_{n=1}^{N-1} \delta_{n}^2 \nu_n \E^{-\delta_{n} t}\right) \leq C \E^{-2\delta_{N-1}t}\left(  N^{4s} \delta_{N-1}^2 + N^s \E^{-c\delta_{N-1} t} \right)\left( N^{4s} \delta_{N-1} + N^{1+s} \E^{-c\delta_{N-1} t} \right).
\end{align*}
Then estimate with $c = 2$,
\begin{align*}
N^{4s} \delta_{N-1}^2 + N^s \E^{-c\delta_{N-1} t} &\leq 4N^{6s-4/3} + N^{s-4} \leq C N^{6s-4/3},\\
N^{4s} \delta_{N-1} + N^s \E^{-c\delta_{N-1} t} &\leq 2N^{4s-2/3} + N^{s-4} \leq C N^{4s-2/3},
\end{align*}
where we used $t \geq 2 \log N/\delta_{N-1}$ and $s \leq 1/44$.  Further, $\E^{-2\delta_{N-1}t} \leq N^{-\alpha} N^{8/3-\alpha+10s} \leq N^{-\alpha-2/3-\sigma+10s}$ as $s \leq \sigma/44$.  Then
\begin{align*}
-E_0(t) \geq N^{-12s-2/3-\alpha}-CN^{-\alpha-2/3-\sigma + 10s},
\end{align*}
provided that this is positive.  Indeed,
\begin{align*}
-E_0(t) \geq N^{-12s-2/3-\alpha}(1-C N^{-\sigma + 22s}) \geq 0,
\end{align*}
for $N$ sufficiently large as $s \leq \sigma/44$.
\end{proof}

Now we look at the leading-order behavior of $E_0(t)$:
\begin{align}\label{e:E0}
E_0(t) =  \frac{1}{4}\delta_{N-1}^2 \nu_{N-1} \E^{-\delta_{N-1}t} \frac{1 + \ds \sum_{n=1}^{N-2} \frac{\delta^2_n}{\delta_{N-1}^2} \frac{\nu_n}{\nu_{N-1}}\E^{-(\delta_n-\delta_{N-1}) t} }{\ds \left( 1 +  \sum_{n=1}^{N-1} \nu_n \E^{-\delta_n t} \right)^2}.
\end{align}
Define $T^*$ by
\begin{align}
\frac{1}{4}\delta_{N-1}^2 \nu_{N-1} \E^{-\delta_{N-1} T^*} &= N^{-\alpha},\notag\\
T^* &= \frac{\alpha \log N + 2 \log \delta_{N-1} + \log \nu_{N-1}-2\log 2}{\delta_{N-1}}.\label{e:t-star}
\end{align}
\begin{lemma}\label{l:t-star} Given Condition~\ref{cond:rigidity}
\begin{align*}
(\alpha -4/3 -4s)\log N/\delta_{N-1} \leq T^* \leq (\alpha -4/3 +4s) \log N/\delta_{N-1}.
\end{align*}
\end{lemma}
\begin{proof}
This follows immediately from the statements
\begin{align*}
N^{-2s} \leq &\nu_{N-1} \leq N^{2s},\\
2N^{-2/3-s} \leq &\delta_{N-1} \leq 2N^{-2/3+s}.
\end{align*}
\end{proof}

Thus, given Condition~\ref{cond:rigidity}, $T^* \in I_\alpha$.  The quantity that we want to estimate is $N^{-2/3}|T-T^*|$.  And we do this by considering the formula
\begin{align*}
E_0(T^{(1)}) - E_0(T^*) = E'_0(\eta)(T^{(1)}-T^*), \quad \text{for some } \eta \in I_\alpha.
\end{align*}
And because $E_0$ is monotone in $I_\alpha$, $E_0(T^{(1)}) = E(T^{(1)}) - E_1(T^{(1)}) = N^{-\alpha} - E_1(T^{(1)})$ we have
\begin{align}\label{e:T-Tstar}
|T^{(1)}-T^*| \leq \frac{|N^{-\alpha} - E_0(T^*) - E_1(T^{(1)})|}{\ds\min_{\eta \in I_\alpha} |E_0'(\eta)|}\leq \frac{|N^{-\alpha} - E_0(T^*)| + \max_{\eta \in I_\alpha}|E_1(\eta)|}{\ds\min_{\eta \in I_\alpha} |E_0'(\eta)|}.
\end{align}
See Figure~\ref{f:H} for a schematic of $E_0, E, T^{(1)}$ and $T^*$.
\begin{figure}[t]
\centerline{\includegraphics[width=.75\linewidth]{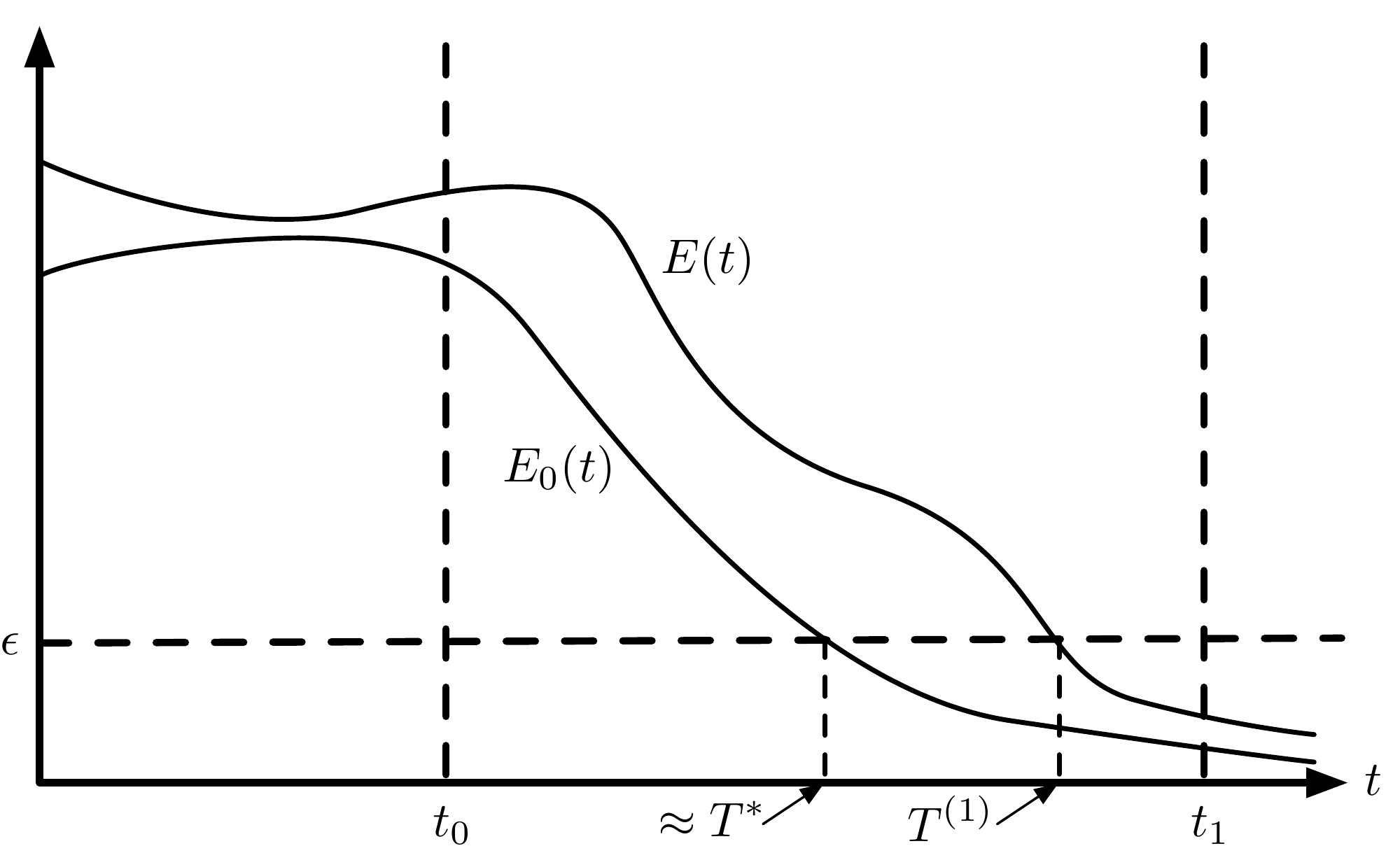}}
\caption{\label{f:H} A schematic for the relationship between the functions $E_0(t)$, $E(t)$ and the times $T^{(1)}$ and $T^*$.  Here $t_0 = (\alpha-4/3-5s) \log N/\delta_{N-1}$ and $t_1 = (\alpha-4/3+7s)\log N/\delta_{N-1}$.  Note that $E_0$ is monotone on $[t_0,t_1]$.}
\end{figure}

Since we already have an adequate estimate on $E_1(T)$ in \eqref{e:H1-est}, it remains to estimate $|N^{-\alpha} - E_0(T^*)|$.
\begin{lemma}\label{l:H0Tstar}
Given Conditions~\ref{cond:gap} and \ref{cond:rigidity}
\begin{align*}
|E_0(T^*) - N^{-\alpha}| \leq CN^{-\alpha-2p+4s}.
\end{align*}
\end{lemma}
\begin{proof}
From \eqref{e:E0} and \eqref{e:t-star} we obtain
\begin{align*}
|E_0(T^*) - N^{-\alpha}| = N^{-\alpha}\frac{ \left|\ds \sum_{n=1}^{N-2} \frac{\delta^2_n}{\delta_{N-1}^2} \frac{\nu_n}{\nu_{N-1}}\E^{-(\delta_n-\delta_{N-1}) T^*} -2 \sum_{n=1}^{N-1} \nu_n \E^{-\delta_n T^*} - \left(\sum_{n=1}^{N-1} \nu_n \E^{-\delta_n T^*}\right)^2 \right|}{\ds \left( 1 +  \sum_{n=1}^{N-1} \nu_n \E^{-\delta_n T^*} \right)^2}.
\end{align*}
We estimate the terms in the numerator individually using the bounds on $T^*$.  For $c =1$, we use that $\alpha-4/3-4s > 2$  and Lemma~\ref{l:estimate} to find
\begin{align*}
\sum_{n=1}^{N-1} \nu_n \E^{-\delta_n T^*} \leq C N^{-\alpha + 4/3 + 4s} \left( N^{4s} + N^{1+s - 2c} \right) \leq CN^{-2 - \sigma + 8s} \leq N^{-2}
\end{align*}
for sufficiently large $N$.  Then we consider the first term the numerator using the index set $I_c$ and Condition~\ref{cond:gap}.  Since our sum is now up to $N-2$ we define $\hat I_c = I_c \cap\{1,\ldots,N-2\}$ and $\hat I_c^c$ to denote the compliment relative to $\{1,\ldots,N-2\}$.  Continuing,
\begin{align*}
\ell(T^*):=\sum_{n=1}^{N-2} \frac{\delta^2_n}{\delta_{N-1}^2} \frac{\nu_n}{\nu_{N-1}}\E^{-(\delta_n-\delta_{N-1}) T^*} = \left(\sum_{n \in \hat I_c^c} +  \sum_{n \in \hat I_c}  \right) \frac{\delta^2_n}{\delta_{N-1}^2} \frac{\nu_n}{\nu_{N-1}}\E^{-(\delta_n-\delta_{N-1}) T^*}
\end{align*}
For $n \in \hat I_c^c$, $\delta_n^2/\delta_{N-1}^2 \leq (1+c)^2$ and
\begin{align*}
\delta_{n}-\delta_{N-1} = 2(\lambda_{N-1}-\lambda_n) \geq 2(\lambda_{N-1}-\lambda_{N-2}) \geq p \delta_{N-1},
\end{align*}
using Condition~\ref{cond:gap}. On the other hand for $n \in \hat I_c$, $\delta_{n} > (1+c) \delta_{N-1}$, and if $c = 3$
\begin{align*}
\delta_{n}-\delta_{N-1} > c\delta_{N-1}  = p\delta_{N-1} + (c-p)\delta_{N-1} \geq p\delta_{N-1} +2\delta_{N-1},
\end{align*}
as $p < 1/3$ and hence $c> 2 + p$.  Using Lemma~\ref{l:Ic} to estimate $|\hat I_c^c|$
\begin{align*}
\sum_{n \in \hat I_c^c} \frac{\delta^2_n}{\delta_{N-1}^2} \frac{\nu_n}{\nu_{N-1}}\E^{-(\delta_n-\delta_{N-1}) T^*} &\leq (1+c)^2 N^{4s} \E^{-p\delta_{N-1} T^*},\\
\sum_{n \in \hat I_c} \frac{\delta^2_n}{\delta_{N-1}^2} \frac{\nu_n}{\nu_{N-1}}\E^{-(\delta_n-\delta_{N-1}) T^*} &\leq [\max_n \delta_n^2]N^{7/3+3s} \E^{-(p + 2) \delta_{N-1}T^*}.
\end{align*}
Given Condition~\ref{cond:rigidity} $[\max_n \delta_n^2] \leq 4(b_V-a_V + 1)^2$ and hence for some $C > 0$, using that $\alpha -4/3-4s > 2$
\begin{align*}
\ell(T^*) &\leq C \E^{-p \delta_{N-1} T^*} \left(N^{4s} + N^{7/3+3s} \E^{- 2\delta_{N-1} T^*}  \right) \\
& \leq C N^{-p(\alpha-4/3-4s)}( N^{4s} + N^{7/3+ 3s - 2(\alpha -4/3 - 4s)})\\
& \leq C N^{-p(\alpha-4/3-4s)}( N^{4s} + N^{-5/3+ 3s})\\
& \leq C N^{-2p + 4s}( 1+ N^{-5/3 -s}).
\end{align*}
Thus
\begin{align*}
\ell(T^*) &\leq C N^{-2p + 4s}.
\end{align*}
From this it follows that
\begin{align*}
|E_0(T^*) - N^{-\alpha}| \leq CN^{-\alpha-2p+4s}.
\end{align*}
\end{proof}
\begin{lemma}\label{l:o1}
Given Conditions~\ref{cond:gap} and \ref{cond:rigidity}, $\sigma$ and $p$ fixed and $s < \min\{\sigma/44,p/8\}$
\begin{align*}
N^{-2/3}|T^{(1)}-T^*| \leq C N^{-2p+16s} \to 0, \quad \text{as } N \to \infty.
\end{align*}
\end{lemma}
\begin{proof}
Combining Lemmas~\ref{l:DH0} and \ref{l:H0Tstar} with \eqref{e:H1-est}, which can be extended to give $E_1(t) \leq N^{-\alpha - 2/3 - \sigma/2}$, and \eqref{e:T-Tstar} we have for sufficiently large $N$
\begin{align*}
N^{-2/3}|T^{(1)}-T^*| \leq CN^{-2/3}N^{\alpha+12s+2/3}\left(N^{-\alpha-2p+4s} +  N^{-\alpha} N^{-2/3-\sigma/2}\right) \leq C\left(N^{-2p+16s} +  N^{-\sigma/2 + 12s}\right),
\end{align*}
where we used $\alpha-8/3 > 2/3$.  Since $p < 1/3$ the right-hand side is bounded by $C N^{-2p + 16s}$ which goes to zero as $N \to \infty$ provided that $s < p/8$, $p < \sigma/4$.
\end{proof}
From \eqref{e:true-error}, we have
\begin{align*}
|\lambda_{N}-X_{11}(t)|  = \half \frac{\ds\sum_{n=1}^{N-1} \delta_n \nu_n \E^{-\delta_nt}}{\ds1+\sum_{n=1}^{N-1} \nu_n \E^{-\delta_nt}} \leq \half\sum_{n=1}^{N-1} \delta_n \nu_n \E^{-\delta_nt}.
\end{align*}
\begin{lemma}\label{l:error}
Given Condition~\ref{cond:rigidity}, $\sigma$ and $p$ fixed and $s < \min\{\sigma/44,p/8\}$
\begin{align*}
\epsilon^{-1}|\lambda_N -X_{11}(T^{(1)})| = N^{\alpha/2}|\lambda_N -X_{11}(T^{(1)})| \leq CN^{-1}
\end{align*}
for sufficiently large $N$.
\end{lemma}
\begin{proof}
We use Lemma~\ref{l:estimate}  with $c = 1$.  By \ref{l:toda-T} we have
\begin{align*}
|\lambda_N -X_{11}(T^{(1)})| \leq  CN^{-\alpha+4/3+5s}(N^{-2/3+5s} +N^{-1+s}) \leq C N^{-\alpha/2} N^{-1}
\end{align*}
because $\alpha - 4/3 -5s\geq 2$.

\end{proof}

\section{Adding probability}\label{s:prob}

We now use the probabilistic facts about Conditions~\ref{cond:rigidity} and \ref{cond:gap} as stated in Theorems~\ref{t:generic} and \ref{t:p} to understand $T^{(1)}$ and $T^*$ as random variables.
\begin{lemma}\label{l:tts}
  For $\alpha \geq 10/3 + \sigma$ and $\sigma > 0$
  \begin{align*}
    \frac{|T^{(1)}-T^*|}{N^{2/3}}
  \end{align*}
  converges to zero in probability as $N \to \infty$.
\end{lemma}
\begin{proof}
  Let $\eps > 0$.  Then
  \begin{align*}
    \mathbb P\left(\frac{|T^{(1)}-T^*|}{N^{2/3}} > \eps \right) = \mathbb P\left(\frac{|T^{(1)}-T^*|}{N^{2/3}} > \eps, G_{N,p} \cap R_{N,s} \right) + \mathbb P\left(\frac{|T^{(1)}-T^*|}{N^{2/3}} > \eps, G_{N,p}^c \cup R_{N,s}^c \right).
  \end{align*}
  If $s$ satisfies the hypotheses in Lemma~\ref{l:o1}, $s < \min\{\sigma/44,p/8\}$, then on the set $G_{N,p} \cap R_{N,s}$, $N^{-2/3}|T-T^*| < \eps$ for $N$ sufficiently large, and hence
  \begin{align*}
\mathbb P\left( \frac{|T^{(1)}-T^*|}{N^{2/3}} > \eps, G_{N,p} \cap R_{N,s} \right) \to 0,
\end{align*}
as $N \to \infty$.  We then estimate
  \begin{align*}
    \mathbb P\left(\frac{|T^{(1)}-T^*|}{N^{2/3}} > \eps, G_{N,p}^c \cup R_{N,s}^c \right) &\leq \mathbb P(G_{N,p}^c) + \mathbb P(R_{N,s}^c),
    \end{align*}
    and by Theorem~\ref{t:generic}
    \begin{align*}
    \limsup_{N \to \infty} \mathbb P\left(\frac{|T^{(1)}-T^*|}{N^{2/3}} > \eps, G_{N,p}^c \cup R_{N,s}^c \right) &\leq \limsup_{N \to \infty}  \mathbb P(G_{N,p}^c).
  \end{align*}
  This is true for any $0< p < 1/3$ and we use Theorem~\ref{t:p}.  So, as $p \downarrow 0$, we find
  \begin{align*}
    \lim_{N \to \infty}\mathbb P\left(\frac{|T^{(1)}-T^*|}{N^{2/3}} > \eps \right) = 0.
  \end{align*}
  \end{proof}

Define
\begin{align}\label{e:t-hat}
  \hat T = \frac{(\alpha - 4/3) \log N}{\delta_{N-1}}.
\end{align}
We need the following simple lemmas in what follows
\begin{lemma}\label{l:whp}
If $X_N {\to} X$ in distribution\footnote{For convergence in distribution, we require that the limiting random variable $X$ satisfies $\mathbb P( |X| < \infty) = 1$.} as $N \to \infty$ then
\begin{align*}
\mathbb P \left( |X_N/a_N|<1 \right) = 1 + o(1)
\end{align*}
as $N \to \infty$ provided that $a_N \to \infty$.
\end{lemma}
\begin{proof}
  For two points of continuity $a,b$ of $F(t) = \mathbb P(X\leq t)$ we have
  \begin{align*}
    \mathbb P\left( a < X_N \leq b \right) \to \mathbb P(a < X \leq b).
  \end{align*}
  Let $M>0$ such that $\pm M$ is a point of continuity of $F$. Then for sufficiently large $N$, $a_N > M$ and
  \begin{align*}
    \liminf_{N \to \infty} \mathbb P( -a_N < X_N < a_N) \geq \liminf_{N \to \infty}\mathbb P( -M < X_N \leq M) = \mathbb P(-M < X \leq M).
  \end{align*}
  Letting $M \to \infty$ we see that $\mathbb P( -a_N \leq X_N \leq a_N) = 1 + o(1)$ as $N \to \infty$.
\end{proof}
\noindent Letting $a_N \to \eta a_N$, $\eta > 0$, we see that the following is true.
\begin{corollary}\label{c:whp}
If $X_N {\to} X$ in distribution as $N \to \infty$ then
\begin{align*}
|X_N/a_N|
\end{align*}
converges to zero in probability provided $a_N \to \infty$.
\end{corollary}
\begin{lemma}\label{l:p2d}
  If as $N \to \infty$, $X_N \to X$ in distribution and $|X_N-Y_N| \to 0$ in probability then $Y_N \to X$ in distribution.
\end{lemma}
\begin{proof}
  Let $t$ be a point of continuity for $\mathbb P(X \leq t)$, then for $\eps > 0$
  \begin{align*}
    \mathbb P( Y_N \leq t) &= \mathbb P(Y_N \leq t, X_N \leq t + \eps) + \mathbb P(Y_N \leq t, X_N > t + \eps)\\
    & \leq \mathbb P(X_N \leq t + \eps ) + \mathbb P(Y_N - X_N \leq t - X_N, t- X_N < - \eps)\\
    &\leq \mathbb P(X_N \leq t + \eps ) + \mathbb P(|Y_N - X_N| > \eps).
  \end{align*}
  Interchanging the roles of $X_N$ and $Y_N$ and replacing $t$ with $t-\eps$ we find
  \begin{align*}
    \mathbb P( X_N \leq t-\eps) \leq \mathbb P(Y_N \leq t) + \mathbb P(|Y_N - X_N| > \eps) \leq  \mathbb P(X_N \leq t + \eps ) + 2\mathbb P(|Y_N - X_N| > \eps).
  \end{align*}
  From this we find that for any $\eta$ such that $t \pm \eta$ are points of continuity
  \begin{align*}
    \mathbb P( X \leq t-\eps) &\leq \liminf_{N \to \infty} \mathbb P(Y_N \leq t) \leq \limsup_{N \to \infty} \mathbb P(Y_N \leq t) \leq \mathbb P( X \leq t+\eps).
  \end{align*}
  By sending $\eps \downarrow 0$ the result follows.
\end{proof}

Now, we compare $T^*$ to $\hat T$.
\begin{lemma}\label{l:tsth}
  For $\alpha \geq 10/3 + \sigma$
  \begin{align*}
    \frac{|T^*-\hat T|}{N^{2/3}\log N}
  \end{align*}
  converges to zero in probability as $N \to \infty$.
\end{lemma}
\begin{proof}
  Consider
  \begin{align*}
    \frac{T^* - \hat T}{N^{2/3}\log N} &= \frac{1}{\log N}  \frac{\log \nu_{N-1} + 2\log N^{2/3}\delta_{N-1}}{N^{2/3}\delta_{N-1}} \\
    &= \frac{1}{\sqrt{\log N}} \left(\frac{1}{(\log N)^{1/4}}|N^{2/3}\delta_{N-1}|^{-1}\right) \left(\frac{2}{(\log N)^{1/4}}\log \nu_{N-1} + \frac{1}{(\log N)^{1/4}}\log N^{2/3}\delta_{N-1} \right).
  \end{align*}
  For
  \begin{align*}
    L_{N} &= \left\{ \frac{1}{(\log N)^{1/4}}|N^{2/3}\delta_{N-1}|^{-1} \leq 1 \right\},\\
    U_{N} &= \left\{ \frac{1}{(\log N)^{1/4}}|\log \nu_{N-1}| \leq  1 \right\},\\
    P_{N} &= \left\{ \frac{1}{(\log N)^{1/4}}|\log N^{2/3}\delta_{N-1}| \leq  1 \right\},
  \end{align*}
we have $\mathbb P(L_N^c) + \mathbb P(U_N^c) + \mathbb P(P_n^c) \to 0$ as $N \to \infty$ by \lem{l:whp} and Theorem~\ref{t:gap-limit}. For these calculations it is important that the limiting distribution function for $N^{2/3}\delta_{N-1}$ is continous at zero, see Theorem~\ref{t:gap-limit}. Then for $\eps > 0$
  \begin{align}\label{e:p-split}\begin{split}
    \mathbb P \left( \left|\frac{T^* - \hat T}{N^{2/3}\log N}\right| > \eps \right) &= \mathbb P \left( \left|\frac{T^* - \hat T}{N^{2/3}\log N}\right| > \eps, L_N \cap U_N  \cap P_N\right)\\
    &+ \mathbb P \left( \left|\frac{T^* - \hat T}{N^{2/3}\log N}\right| > \eps, L_N^c \cup U_N^c \cup P_N^c \right).
  \end{split}\end{align}
  On the set $L_N \cap U_N \cap P_N$ we estimate
  \begin{align*}
  \left|\frac{T^* - \hat T}{N^{2/3}\log N}\right| \leq \frac{3}{\sqrt{\log N}}.
  \end{align*}
  Hence first term on the right-hand side of \eqref{e:p-split} is zero for sufficiently large $N$ and the second term is bounded by $\mathbb P(U_N^c) + \mathbb P(L_N^c) + \mathbb P(P_N^c)$ which tends to zero.  This shows convergence in probability.
\end{proof}
We now arrive at our main result.
\begin{theorem}\label{t:main-alpha}
  If $\alpha \geq 10/3 + \sigma$ and $\sigma > 0$ then
  \begin{align*}
    \lim_{N\to \infty} \mathbb P\left( \frac{2^{2/3}T^{(1)}}{c_V^{2/3}(\alpha-4/3) N^{2/3} \log N} \leq t\right) =  F^{\mathrm{gap}}_\beta(t).
  \end{align*}
\end{theorem}
\begin{proof}
Combining \lem{l:tts} and \lem{l:tsth} we have that
\begin{align*}
\left|2^{2/3}\frac{T^{(1)}-\hat T}{c_V^{2/3}(\alpha-4/3) N^{2/3}\log N}\right|
\end{align*}
converges to zero in probability.  Then by \lem{l:p2d} and Theorem~\ref{t:gap-limit} the result follows as
\begin{align*}
\lim_{N \to \infty} \mathbb P \left( \frac{2^{2/3} \hat T}{c_V^{2/3}(\alpha-4/3) N^{2/3} \log N} \right) = \lim_{N \to \infty} \mathbb P ( c_V^{-2/3}2^{2/3}N^{-2/3} (\lambda_N-\lambda_{N-1})^{-1} \leq t ) = F^{\mathrm{gap}}_\beta(t).
\end{align*}
\end{proof}

We also prove a result concerning the true error $|\lambda_N - X_{11}(T^{(1)})|$:
  \begin{proposition}\label{p:error-alpha}
  For $\alpha \geq 10/3 + \sigma$ and $\sigma > 0$ and any $q < 1$
  \begin{align*}
    N^{\alpha/2 + q}|\lambda_N - X_{11}(T^{(1)})|
  \end{align*}
  converges to zero in probability as $N \to \infty$.  Furthermore, for any $r > 0$
  \begin{align*}
  N^{2/3+r}|\gamma_N - X_{11}(T^{(1)})|,\quad N^{2/3+r}|\lambda_j - X_{11}(T^{(1)})|,
  \end{align*}
  converges to $\infty$ in probability, if $j = j(N) < N$.
  \end{proposition}
  \begin{proof}
    We recall that $R_{N,s}$ is the set on which Condition~\ref{cond:rigidity} holds.  Then for any $\eta > 0$
    \begin{align*}
      \mathbb P(N^{\alpha/2 + q}|\lambda_N &- X_{11}(T^{(1)})| > \eta) \\
      &= \mathbb P(N^{\alpha/2 + q}|\lambda_N - X_{11}(T^{(1)})| > \eta, R_{N,s}) + \mathbb P(N^{\alpha/2 + q}|\lambda_N - X_{11}(T^{(1)})| > \eta, R^c_{N,s})\\
      & \leq \mathbb P(N^{\alpha/2 + q}|\lambda_N - X_{11}(T^{(1)})| > \eta, R_{N,s}) + \mathbb P(R_{N,s}^c).
    \end{align*}
    Using Lemma~\ref{l:error}, the first term on the right-hand side is zero for sufficiently large $N$ and the second term vanishes from Theorem~\ref{t:generic}.  This shows the first statement, \emph{i.e.},
    \begin{align*}
      \lim_{N \to \infty} \mathbb P(N^{\alpha/2 + q}|\lambda_N &- X_{11}(T^{(1)})| > \eta) = 0.
    \end{align*}

    For the second statement, on the set $R_{N,s}$ with $s < \min\{r,\sigma/44,p/8\}$ we have
  \begin{align*}
  	|\lambda_j - X_{11}(T^{(1)})| \geq |\lambda_j - \lambda_N| - |\lambda_N-X_{11}(T^{(1)})| \geq |\lambda_{N-1} - \lambda_N| - |\lambda_N-X_{11}(T^{(1)})|,
  \end{align*}
  and for sufficiently large $N$ (see Lemma~\ref{l:error})
  \begin{align*}
  N^{2/3+r}|\lambda_j - X_{11}(T^{(1)})| \geq N^{r} (N^{2/3}|\lambda_{N-1} - \lambda_N| - N^{-1/3-\alpha/2}) \geq N^{r-s} (1 - CN^{-1/3-\alpha/2+s}).
  \end{align*}
  This tends to $\infty$ as $s< 1/3$ and $s < r$.   Hence for any $K> 0$, again using the arguments of Theorem~\ref{t:main-alpha},
  \begin{align*}
    \mathbb P &\left( N^{2/3+r}|\lambda_j - X_{11}(T^{(1)})| > K \right)\\
    & =  \mathbb P \left( N^{2/3+r}|\lambda_j - X_{11}(T^{(1)})| > K, R_{N,s} \right) +  \mathbb P \left( N^{2/3+r}|\lambda_j - X_{11}(T^{(1)})| > K, R_{N,s}^c \right).
  \end{align*}
  For sufficiently large $N$, the first term on the right-hand side is equal to $\mathbb P(R_{N,s})$ and the second term is bounded by $\mathbb P(R^c_{N,s})$ and hence
  \begin{align*}
\lim_{N\to \infty} \mathbb P &\left( N^{2/3+r}|\lambda_j - X_{11}(T^{(1)})| > K \right) =1.
  \end{align*}
  Next, under the same assumption (Condition~\ref{cond:rigidity})
  \begin{align*}
  N^{2/3+r}|b_V - X_{11}(T^{(1)})| \geq N^{r} (N^{2/3}|b_V - \lambda_N| - CN^{-1/3-\alpha/2}).
  \end{align*}
  From Corollary~\ref{c:whp} and Theorem~\ref{t:gap-limit} using $\gamma_N = b_V$
  \begin{align*}
  N^{-r}(N^{2/3}|b_V - \lambda_N| - CN^{-1/3-\alpha/2})^{-1}
  \end{align*}
  converges to zero in probability (with no point mass at zero), implying its inverse converges to $\infty$ in probability.  This shows $N^{\alpha}|b_V - X_{11}(T^{(1)})|$ converges to $\infty$ in probability.
  \end{proof}

  \section*{Acknowledgments}

  The authors would like to thank Yuri Bakhtin and Paul Bourgade for many useful conversations and the anonymous referee for valuable suggestions.  We also thank Folkmar Bornemann for the data for $F_2^\mathrm{gap}$.  This work was supported in part by grants NSF-DMS-1303018 (TT) and NSF-DMS-1300965 (PD).
  
  \appendix
  
  \section{Invariant and Wigner ensembles}\label{s:ensemble}
  
 The following definitions are taken from \cite{Bourgade2013,Erdos2013,DeiftOrthogonalPolynomials}.  \blue{The first definition appeared initially in \cite{Erdos2012a} and was made more explicit in \cite{Erdos2013}.}  These are the two classes of random matrices to which our results apply.

\begin{definition}[Generalized Wigner Ensemble (WE)]\label{def:WE}
A generalized Wigner matrix (ensemble) is a real symmetric ($\beta = 1$) or Hermitian ($\beta =2$) matrix $H = (H_{ij})_{i,j=1}^N$ such that $H_{ij}$ are independent random variables for $i \leq j$ given by a probability measure $\nu_{ij}$ with
\begin{align*}
\mathbb E H_{ij} = 0, \quad \sigma_{ij}^2 := \mathbb E H_{ij}^2.
\end{align*}
Next, assume there is a fixed constant $v$ (independent of $N,i,j$) such that
\begin{align*}
\mathbb P(|H_{ij}| > x \sigma_{ij}) \leq v^{-1} \exp(-x^v),\quad x > 0.
\end{align*}
Finally, assume there exists $C_1,C_2>0$ such that for all $i,j$
\begin{align*}
\sum_{i=1}^N \sigma_{ij}^2 = 1,\quad
\frac{C_1}{N} \leq \sigma_{ij}^2 \leq \frac{C_2}{N},
\end{align*}
and for $\beta = 2$ the matrix
\begin{align*}
\Sigma_{ij} = \begin{mat} \mathbb E(\real H_{ij})^2 & \mathbb E(\real H_{ij})(\imag H_{ij})\\
\mathbb E(\real H_{ij})(\imag H_{ij}) & \mathbb E(\imag H_{ij})^2 \end{mat}
\end{align*}
has its smallest eigenvalue $\lambda_{\min}$ satisfy $\lambda_{\min} \geq C_1 N^{-1}$.
\end{definition}

\begin{definition}[Invariant Ensemble (IE)]\label{def:IE}
Let $V: \mathbb R \to \mathbb R$ satisfy $V \in C^4(\mathbb R)$, $\inf_{x \in \mathbb R} V''(x) > 0$ and $V(x) > (2 + \delta) \log (1 + |x|)$ for sufficiently large $x$ and some fixed $\delta > 0$.  Then we define an invariant ensemble\footnote{This is not the most general class of $V$ but these assumptions simplify the analysis.} to be the set of all $N\times N$ symmetric ($\beta = 1$) or Hermitian ($\beta = 2$) matrices $H = (H_{ij})_{i,j =1}^N$ with probability density
\begin{align*}
\frac{1}{Z_{N}}\E^{-N\frac{\beta}{2} \mathrm{tr} V(H)} \D H
\end{align*}
Here $\D H = \prod_{i \leq j} \D H_{ij}$ if $\beta = 1$ and $\D H = \prod_{i=1}^N \D H_{ii} \prod_{i < j} \D \real H_{ij} \D\imag H_{ij}$ if $\beta = 2$.

\end{definition}

\bibliographystyle{apalike}
\bibliography{library}

\begin{thebibliography}{}

\bibitem[Adler, 1978]{Adler1978}
Adler, M. (1978).
\newblock {On a trace functional for formal pseudo-differential operators and
  the symplectic structure of the Korteweg-devries type equations}.
\newblock {\em Invent. Math.}, 50(3):219--248.

\bibitem[Bourgade, 2016]{BourgadeConv}
Bourgade, P. (2016).
\newblock {Personal communication}.

\bibitem[Bourgade et~al., 2014]{Bourgade2014}
Bourgade, P., Erdős, L., and Yau, H.-T. (2014).
\newblock {Edge Universality of Beta Ensembles}.
\newblock {\em Commun. Math. Phys.}, 332(1):261--353.

\bibitem[Bourgade and Yau, 2013]{Bourgade2013}
Bourgade, P. and Yau, H.-T. (2013).
\newblock {The Eigenvector Moment Flow and local Quantum Unique Ergodicity}.
\newblock pages 1--35.

\bibitem[Deift, 2000]{DeiftOrthogonalPolynomials}
Deift, P. (2000).
\newblock {\em {Orthogonal Polynomials and Random Matrices: a Riemann-Hilbert
  Approach}}.
\newblock Amer. Math. Soc., Providence, RI.

\bibitem[Deift and Gioev, 2007]{Deift2007a}
Deift, P. and Gioev, D. (2007).
\newblock {Universality at the edge of the spectrum for unitary, orthogonal,
  and symplectic ensembles of random matrices}.
\newblock {\em Commun. Pure Appl. Math.}, 60(6):867--910.

\bibitem[Deift et~al., 1986]{Deift1986a}
Deift, P., Li, L.~C., Nanda, T., and Tomei, C. (1986).
\newblock {The Toda flow on a generic orbit is integrable}.
\newblock {\em Commun. Pure Appl. Math.}, 39(2):183--232.

\bibitem[Deift et~al., 1985]{Deift1985}
Deift, P., Li, L.~C., and Tomei, C. (1985).
\newblock {Toda flows with infinitely many variables}.
\newblock {\em J. Funct. Anal.}, 64(3):358--402.

\bibitem[Deift et~al., 1983]{DeiftEigenvalue}
Deift, P., Nanda, T., and Tomei, C. (1983).
\newblock {Ordinary differential equations and the symmetric eigenvalue
  problem}.
\newblock {\em SIAM J. Numer. Anal.}, 20:1--22.

\bibitem[Deift and Trogdon, 2017]{Deift2017}
Deift, P. and Trogdon, T. (2017).
\newblock {Universality for eigenvalue algorithms on sample covariance
  matrices}.

\bibitem[Deift et~al., 2014]{Deift2014}
Deift, P.~A., Menon, G., Olver, S., and Trogdon, T. (2014).
\newblock {Universality in numerical computations with random data}.
\newblock {\em Proc. Natl. Acad. Sci. U. S. A.}, 111(42):14973--8.

\bibitem[Erdos et~al., 2013]{Erdos2013}
Erdos, L., Knowles, A., Yau, H.-T., and Yin, J. (2013).
\newblock {The local semicircle law for a general class of random matrices}.
\newblock {\em Electron. J. Probab.}, 18.

\bibitem[Erdős, 2012]{Erdos2012b}
Erdős, L. (2012).
\newblock {Universality for random matrices and log-gases}.
\newblock In Jerison, D., Kisin, M., Mrowka, T., Stanley, R., Yau, H.-T., and
  Yau, S.-T., editors, {\em Curr. Dev. Math.}, number 2010.

\bibitem[Erdős et~al., 2012]{Erdos2012a}
Erdős, L., Yau, H.-T., and Yin, J. (2012).
\newblock {Rigidity of eigenvalues of generalized Wigner matrices}.
\newblock {\em Adv. Math. (N. Y).}, 229(3):1435--1515.

\bibitem[Flaschka, 1974]{Flaschka1974}
Flaschka, H. (1974).
\newblock {The Toda lattice. I. Existence of integrals}.
\newblock {\em Phys. Rev. B}, 9(4):1924--1925.

\bibitem[Jiang, 2006]{Jiang2006}
Jiang, T. (2006).
\newblock {How many entries of a typical orthogonal matrix can be approximated
  by independent normals?}
\newblock {\em Ann. Probab.}, 34(4):1497--1529.

\bibitem[Kostant, 1979]{Kostant1979}
Kostant, B. (1979).
\newblock {The solution to a generalized Toda lattice and representation
  theory}.
\newblock {\em Adv. Math. (N. Y).}, 34(3):195--338.

\bibitem[Manakov, 1975]{Manakov1975}
Manakov, S.~V. (1975).
\newblock {Complete integrability and stochastization of discrete dynamical
  systems.}
\newblock {\em Sov. Phys. JETP}, 40(2):269--274.

\bibitem[Monthus and Garel, 2013]{Monthus2013}
Monthus, C. and Garel, T. (2013).
\newblock {Typical versus averaged overlap distribution in spin glasses:
  Evidence for droplet scaling theory}.
\newblock {\em Phys. Rev. B}, 88(13):134204.

\bibitem[Moser, 1975]{Moser1975}
Moser, J. (1975).
\newblock {Three integrable Hamiltonian systems connected with isospectral
  deformations}.
\newblock {\em Adv. Math. (N. Y).}, 16(2):197--220.

\bibitem[Perret and Schehr, 2014]{Perret2013}
Perret, A. and Schehr, G. (2014).
\newblock {Near-Extreme Eigenvalues and the First Gap of Hermitian Random
  Matrices}.
\newblock {\em J. Stat. Phys.}, 156(5):843--876.

\bibitem[Pfrang et~al., 2014]{DiagonalRMT}
Pfrang, C.~W., Deift, P., and Menon, G. (2014).
\newblock {How long does it take to compute the eigenvalues of a random
  symmetric matrix?}
\newblock {\em Random matrix theory, Interact. Part. Syst. Integr. Syst. MSRI
  Publ.}, 65:411--442.

\bibitem[Ram{\'{i}}rez et~al., 2011]{Ramirez2011}
Ram{\'{i}}rez, J.~A., Rider, B., and Vir{\'{a}}g, B. (2011).
\newblock {Beta ensembles, stochastic Airy spectrum, and a diffusion}.
\newblock {\em J. Am. Math. Soc.}, 24(4):919--944.

\bibitem[Soshnikov, 1999]{Soshnikov1999}
Soshnikov, A. (1999).
\newblock {Universality at the Edge of the Spectrum in Wigner Random Matrices}.
\newblock {\em Commun. Math. Phys.}, 207(3):697--733.

\bibitem[Stam, 1982]{Stam1982}
Stam, A.~J. (1982).
\newblock {Limit theorems for uniform distributions on spheres in
  high-dimensional Euclidean spaces}.
\newblock {\em J. Appl. Probab.}, 19(1):221--228.

\bibitem[Symes, 1982]{Symes1982}
Symes, W.~W. (1982).
\newblock {The QR algorithm and scattering for the finite nonperiodic Toda
  lattice}.
\newblock {\em Phys. D Nonlinear Phenom.}, 4(2):275--280.

\bibitem[Tao and Vu, 2010]{Tao2010}
Tao, T. and Vu, V. (2010).
\newblock {Random Matrices: Universality of Local Eigenvalue Statistics up to
  the Edge}.
\newblock {\em Commun. Math. Phys.}, 298(2):549--572.

\bibitem[Toda, 1967]{Toda1967}
Toda, M. (1967).
\newblock {Vibration of a chain with nonlinear interaction}.
\newblock {\em J. Phys. Soc. Japan}, 22(2):431--436.

\bibitem[Tracy and Widom, 1994]{TracyWidom}
Tracy, C.~A. and Widom, H. (1994).
\newblock {Level-spacing distributions and the Airy kernel}.
\newblock {\em Comm. Math. Phys.}, 159:151--174.

\bibitem[Watkins, 1984]{WatkinsIsospectral}
Watkins, D.~S. (1984).
\newblock {Isospectral flows}.
\newblock {\em SIAM Rev.}, 26(3):379--391.

\bibitem[Witte et~al., 2013]{Witte2013}
Witte, N.~S., Bornemann, F., and Forrester, P.~J. (2013).
\newblock {Joint distribution of the first and second eigenvalues at the soft
  edge of unitary ensembles}.
\newblock {\em Nonlinearity}, 26(6):1799--1822.

\end{thebibliography}

\end{document}